\documentclass[a4paper,draft]{ZbRad}
\usepackage{amssymb}
\usepackage{bbm,enumitem}
\usepackage{dutchcal} 
\usepackage{todonotes} 
\usepackage{amscd} 

\theoremstyle{plain}
 \newtheorem{theorem}{Theorem}[section]
 \newtheorem{proposition}[theorem]{Proposition}
 \newtheorem{lemma}[theorem]{Lemma}
 \newtheorem{corollary}[theorem]{Corollary}
 \newtheorem*{claim*}{Claim}
\theoremstyle{definition}
 
 \newtheorem{remark}[theorem]{Remark}
   
  \newtheorem{question}{Question}
 \newtheorem{definition}[theorem]{Definition}
\numberwithin{equation}{section}

\newcommand{\anf}[1]{{\text{``}\hspace{0.3ex}{#1}\hspace{0.3ex}\text{''}}}
\newcommand{\cof}[1]{{{\rm{cof}}(#1)}}
\newcommand{\dom}[1]{{{\rm{dom}}(#1)}}
\newcommand{\ran}[1]{{{\rm{ran}}(#1)}}
\newcommand{\length}[2]{{{\rm{lh}}_{{#2}}(#1)}}
\newcommand{\Set}[2]{\{{#1}~\vert~{#2}\}}
\newcommand{\seq}[2]{\langle{#1}~\vert~{#2}\rangle}
\newcommand{\map}[3]{{#1}:{#2}\longrightarrow{#3}}
\newcommand{\Map}[5]{{#1}:{#2}\longrightarrow{#3};~{#4}\longmapsto{#5}}
\newcommand{\POT}[1]{{\mathcal{P}}({#1})}
\newcommand{\Club}[1]{{\mathcal{Club}}({#1})}
\newcommand{\NS}[1]{{\mathcal{NS}}({#1})}
\newcommand{\Cond}[1]{{\mathsf{Cond}}({#1})}
\newcommand{\HH}[1]{{\rm{H}}(#1)}
\newcommand{\Refl}[1]{{\rm{Refl}}(#1)}
\newcommand{\Lim}{{\rm{Lim}}}
\newcommand{\ZFC}{{\rm{ZFC}}}

\newcommand{\GCH}{{\rm{GCH}}}
\newcommand{\MM}{{\rm{MM}}}
\newcommand{\BMM}{{\rm{BMM}}}
\newcommand{\PFA}{{\rm{PFA}}}
\newcommand{\Ord}{{\rm{Ord}}}
\newcommand{\Add}[2]{{\rm{Add}}({#1},{#2})}

\newcommand{\PPP}{\mathbb{P}}
\newcommand{\SSS}{\mathbb{S}}
\newcommand{\TTT}{\mathbb{T}}

\newcommand{\calA}{\mathcal{A}}
\newcommand{\calB}{\mathcal{B}}
\newcommand{\calC}{\mathcal{C}}
\newcommand{\calD}{\mathcal{D}}
\newcommand{\calE}{\mathcal{E}}

\newcommand{\calI}{\mathcal{I}}

\newcommand{\calN}{\mathcal{N}}

\newcommand{\calS}{\mathcal{S}}
\newcommand{\calT}{\mathcal{T}}
\newcommand{\calU}{\mathcal{U}}
\newcommand{\calX}{\mathcal{X}}

\renewcommand{\leq}{\leqslant}
\renewcommand{\geq}{\geqslant}
\renewcommand{\setminus}{\smallsetminus}
\newfont{\TITf}{cmssdc10 scaled 1440}

\title{The complexity of non-stationary ideals}

\author{Philipp L\"{u}cke}
\address{Fachbereich Mathematik, Universit\"at Hamburg, Bundesstra{\ss}e 55, Hamburg, 20146, Germany}
\email{philipp.luecke@uni-hamburg.de}

\begin{document}

\setcounter{page}{1}

\vspace*{40mm}

\thispagestyle{empty}

{
\TITf\setlength{\parskip}{\smallskipamount}

\begin{center}
Philipp L\"{u}cke

\bigskip\bigskip\bigskip

{THE COMPLEXITY OF NON-STATIONARY IDEALS}

\end{center}
}

\vspace{6em}{\leftskip3em\rightskip3em
\emph{Abstract}.
We present an overview of results on the question of whether the non-stationary ideal of an uncountable regular cardinal $\kappa$ can be defined by a $\Pi_1$-formula using parameters of hereditary cardinality at most $\kappa$. 
These results   show that this question is deeply connected to several central topics  of current research in set theory.



}

\newpage\thispagestyle{empty}

\maketitle
\tableofcontents


\section{Introduction}  

A central aspect of the combinatorial behavior of uncountable regular cardinals $\kappa$ is given by the fact that the collection $\Club{\kappa}$ of all subsets of $\kappa$ that contain a closed unbounded subset forms a normal filter on $\kappa$, the \emph{club filter} on $\kappa$.  
 Since the structural properties of these filters and the corresponding dual ideals $\NS{\kappa}$ of non-stationary subsets provide important information about the underlying model of set theory, the study of these objects plays a central role in modern set theory. 
 In \cite{MR1222536} and \cite{MR1242054}, Mekler, Shelah and V{\"a}{\"a}n{\"a}nen initiated the study of the complexity of club filters and non-stationary ideals and showed that these investigations are deeply connected to several research lines  in both model theory and set theory.

 Remember that a formula $\varphi$ in the language of set theory is a \emph{$\Sigma_0$-formula} if it is contained in the smallest collection of formulas in this language that contains all atomic formulas and is closed under negations, conjunctions, disjunctions  and bounded quantification. Next, for a natural number $n$, we say that the negation of a $\Sigma_n$-formula is a \emph{$\Pi_n$-formula}. Finally, given a natural number $n$, a set-theoretic formula is a \emph{$\Sigma_{n+1}$-formula} if it is of the form $\exists x \psi$ for some $\Pi_n$-formula $\Psi$. 
 In the following, we say that a class $A$ is \emph{definable} by a formula $\varphi(v_0,v_1)$ and a parameter $b$ if $A=\Set{a}{\varphi(a,b)}$ holds. 
 Given an uncountable regular cardinal $\kappa$, it is now easy to see that the sets $\Club{\kappa}$ and $\NS{\kappa}$ are both definable by $\Sigma_1$-formulas with parameter $\kappa$. 
 In contrast, it is not clear if it is also possible to define these sets by a $\Pi_1$-formula and parameters contained in the collection $\HH{\kappa^+}$ of all sets of hereditary cardinality at most $\kappa$, and   the study of the complexity of $\Club{\kappa}$ and $\NS{\kappa}$ focusses on answering this question for various uncountable regular cardinals $\kappa$ in different models of set theory. 
 For this purpose, given an infinite cardinal $\kappa$, we say that a set $\calA$ of subsets of $\kappa$ is a \emph{$\mathbf{\Delta}_1$-subset of $\POT{\kappa}$} if $\calA$ is definable by both a $\Sigma_1$- and a $\Pi_1$-formula with parameters in $\HH{\kappa^+}$. Using this terminology, we can now phrase the above question in the following way:

\begin{question}\label{question:Q1}
 Given an uncountable regular cardinal $\kappa$, is the set $\Club{\kappa}$ a $\mathbf{\Delta}_1$-subset of $\POT{\kappa}$? 
\end{question}

 Note that the basic closure properties\footnote{See, for example, {\cite[Chapter 1]{MR750828}}.} of the classes of $\Sigma_1$- and $\Pi_1$-definable subsets of $\POT{\kappa}$ ensure that the above question has an  affirmative answer if and only if $\NS{\kappa}$ is a $\mathbf{\Delta}_1$-subset of $\POT{\kappa}$. 
 Motivated by the fact that, in our results providing negative answers to the above questions, we often derive a statement that substantially strengthens the non-$\mathbf{\Delta}_1$-definability of $\Club{\kappa}$, we also consider a weakening of Question \ref{question:Q1} whose formulation is motivated by the classical \emph{Lusin Separation Theorem} (see {\cite[Theorem 14.7]{MR1321597}}) from descriptive set theory. 
  Given a set $\calX$ and disjoint subsets $\calA$ and $\calB$ of $\calX$, we say that a subset $\calS$ of $\calX$ \emph{separates $\calA$ from $\calB$} if $\calA\subseteq\calS\subseteq\calX\setminus\calB$ holds.

\begin{question}\label{question:Q2}
 Given an uncountable regular cardinal $\kappa$, is there a $\mathbf{\Delta}_1$-subset of $\POT{\kappa}$ that separates $\Club{\kappa}$ from $\NS{\kappa}$? 
\end{question}

 Typical candidates for subsets of $\POT{\kappa}$ that provide affirmative answers to Question \ref{question:Q2} are given by the restrictions of the club filter  to stationary subsets of $\kappa$. 
 Given an uncountable regular cardinal $\kappa$ and a stationary subset subset $S$ of $\kappa$, we define $$\Club{S}  =  \Set{A\subseteq\kappa}{A\cup(\kappa\setminus S)\in\Club{\kappa}}$$ as well as $$\NS{S}  =  \Set{A\subseteq\kappa}{\kappa\setminus A\in\Club{S}}.$$
 Note that $\Club{S}$ and $\NS{S}$ are disjoint subsets of $\POT{\kappa}$ that are definable by $\Sigma_1$-formulas with parameter $S$. Moreover, the above definition ensures that the set $\Club{S}$ is a $\mathbf{\Delta}_1$-subset of $\POT{\kappa}$ if and only if the set $\NS{S}$ has this property. Finally, it is easy to see that the stationarity of $S$ guarantees that the set $\Club{S}$ separates $\Club{\kappa}$ from $\NS{\kappa}$. This shows that the existence of a stationary subset $S$ of $\kappa$ with the property that $\Club{S}$ is a $\mathbf{\Delta}_1$-subset of $\kappa$ provides an affirmative answer to Question \ref{question:Q2}.

  In the remainder of this paper, we will present results that provide answers to the above questions in different models of set theory. These reveal deep connections between these questions and  central topics  of contemporary research in set theory. 
  In Section \ref{section:2}, we will discuss four settings that provide affirmative answers to the above questions. Contrasting this, the results discussed in Section \ref{section:3} present four settings in which well-studied structural properties of the models of set theory lead to negative answers to these questions.


\section{Definability results}\label{section:2}

In the following, we present different  ways to obtain affirmative answers to the  above two questions.


\subsection{Dense ideals}
 Our first and most direct examples of   $\mathbf{\Delta}_1$-definable club filters arise from strong saturation properties of non-stationary ideals.

\begin{definition}
 An ideal $\calI$ on an infinite cardinal $\kappa$ is \emph{dense} if there exists a subset $\calD$ of $\POT{\kappa}\setminus\calI$ of cardinality $\kappa$ with the property that for every $A\in\POT{\kappa}\setminus\calI$, there exists $D\in\calD$ with $D\setminus A\in\calI$. 
\end{definition}

Note that a dense ideal $\calI$ on a cardinal $\kappa$ is \emph{$\kappa^+$-saturated}, {i.e.,} there exists no sequence $\seq{S_\gamma}{\gamma<\kappa^+}$ of elements of $\POT{\kappa}\setminus\calI$ with the property that $S_\gamma\cap S_\delta\in\calI$ holds for all $\gamma<\delta<\kappa^+$. We are now interested in the complexity of dense ideals.

\begin{proposition}
 Let $\calI$ be a dense ideal on an infinite  cardinal $\kappa$. If $\calI$ is definable by a $\Sigma_1$-formula with parameters in $\HH{\kappa^+}$, then  $\calI$ is a $\mathbf{\Delta}_1$-subset of $\POT{\kappa}$. 
\end{proposition}

\begin{proof}
 Pick $\calD\subseteq\POT{\kappa}\setminus\calI$ witnessing that $\calI$ is dense. 
  Given $A\subseteq\kappa$, we have $A\notin\calI$ if and only if there is $D\in\calD$ with $D\setminus A\in\calI$. 
  Since $\calD\in\HH{\kappa^+}$ and $\calI$ is  definable by a $\Sigma_1$-formula with parameters in $\HH{\kappa^+}$, this equivalence shows that $\POT{\kappa}\setminus\calI$ is also  definable by a $\Sigma_1$-formula with parameters in $\HH{\kappa^+}$. 
\end{proof}

\begin{corollary}
 If $\NS{\omega_1}$ is dense, then $\Club{\omega_1}$ is a $\mathbf{\Delta}_1$-subset of $\POT{\omega_1}$. \qed 
\end{corollary}

 Note that Woodin proved in \cite{MR1713438}  that, over the theory $\ZFC$, the statement that $\NS{\omega_1}$ is dense is equiconsistent to the existence of infinitely many Woodin cardinals. 
 Moreover,  a combination of Theorem \ref{theorem:HoffelnerEtAlDelta} below with the results of \cite{MR4250390} and \cite{MM} shows that the assumption that $\NS{\omega_1}$ is $\aleph_2$-saturated does not imply that $\Club{\omega_1}$ is a $\mathbf{\Delta}_1$-subset of $\POT{\omega_1}$.


\subsection{Canary trees} 
  Next, we discuss the historically first examples of $\mathbf{\Delta}_1$-definable restrictions of the club filters to stationary sets that arose from the study of so-called \emph{canary trees} by Mekler and Shelah in \cite{MR1222536}, and later work of Hyttinen and Rautila in \cite{MR1877015}. 
 Given infinite regular cardinals $\mu<\kappa$, we let $S^\kappa_\mu$ denote the set of all limit ordinals $\alpha<\kappa$ of cofinality $\mu$.
   Moreover, for   natural numbers $m<n$, we write $S^n_m$ instead of $S^{\omega_n}_{\omega_m}$. 
  In addition,    we define $S^\kappa_{{<}\mu}$ and $S^\kappa_{{>}\mu}$ in the obvious ways. 
     Finally, given a (set-theoretic) tree $\TTT$, we let $[\TTT]$ denote the set of cofinal branches through $\TTT$.

\begin{definition}[\cite{MR1877015,MR1222536}]
 Given an infinite regular cardinal $\kappa$, a \emph{$\kappa$-canary tree} is a tree $\TTT$ of height $\kappa^+$ with $[\TTT]=\emptyset$ and the property that $\mathbbm{1}_\PPP\Vdash\anf{[\check{\TTT}]\neq\emptyset}$ holds whenever $S$ is a stationary  subset of $S^{\kappa^+}_\kappa$ and $\PPP$ is a ${<}\kappa^+$-distributive partial order with $\mathbbm{1}_\PPP\Vdash\anf{\check{S}\in\NS{\check{\kappa}^+}}$. 
\end{definition}

In the following, we  present results from \cite{MR1242054} that relate the existence and non-existence of $\kappa$-canary trees to the complexity of the club filter on $\kappa$. The starting point to establish this connection is given by the next definition:

 \begin{definition}
   Given an infinite regular cardinal $\kappa$ and a subset $S$ of $S^{\kappa^+}_\kappa$, we let $\TTT(S)$ denote the tree whose underlying set consists of all strictly increasing $s\in{}^{{<}\kappa^+}\kappa^+$ such that $\dom{s}$ is a successor ordinal, $\ran{s}\subseteq S$ and $s$ is continuous at all points of cofinality $\kappa$ in its domain and whose ordering is given by end-extension. 
 \end{definition}

 \begin{proposition}\label{proposition:TreesTS}
  Let $\kappa$ be an infinite regular cardinal. 
  \begin{enumerate}
   \item\label{item:TreeStat1} If $S$ is a stationary subset of $S^{\kappa^+}_\kappa$, then the tree $\TTT(S)$ has height $\kappa^+$. 
   
   \item\label{item:TreeStat2} A subset $S$ of $S^{\kappa^+}_\kappa$ is an element of $\Club{S^{\kappa^+}_\kappa}$ if and only if the tree $\TTT(S)$ has height $\kappa^+$ and $[\TTT(S)]\neq\emptyset$ holds. 
  \end{enumerate}
 \end{proposition}
 
 \begin{proof}
   \eqref{item:TreeStat1} For every closed unbounded subset $C$ of $\kappa^+$, the set $C\cap(S\cup S^{\kappa^+}_{{<}\kappa})$ contains a closed subset of order-type $\kappa+1$. Therefore, {\cite[Lemma 1.12]{MR0716625}} shows that for every $\gamma<\kappa^+$, the set $S\cup S^{\kappa^+}_{{<}\kappa}$ contains a closed subset of order-type $\gamma+1$. By considering the monotone enumerations of the intersections of such closed subsets with $S$, we can now conclude that the tree $\TTT(S)$ has height $\kappa^+$. 
   
   \eqref{item:TreeStat2}  First, let $S$ be a subset of $S^{\kappa^+}_\kappa$   in     $\Club{S^{\kappa^+}_\kappa}$, let $C$ be a closed unbounded subset of $\kappa^+$ with $C\cap S^{\kappa^+}_\kappa\subseteq S$ and let $\map{s}{\kappa^+}{C\cap S}$ be the monotone enumeration of $C\cap S$. We then know that for every $\gamma<\kappa^+$, the function $s\restriction(\gamma+1)$ is an element of $\TTT(S)$ with $\length{s\restriction(\gamma+1)}{\TTT(S)}=\gamma$. This directly shows that $\TTT(S)$ is a tree of height $\kappa^+$ with $[\TTT(S)]\neq\emptyset$. 
   
   Now, assume that $S$ is a subset of $S^{\kappa^+}_\kappa$ with the property that the tree $\TTT(S)$ has height  $\kappa^+$ and $[\TTT]\neq\emptyset$. Then there is a function $\map{s}{\kappa^+}{\kappa^+}$ with the property that $s\restriction(\gamma+1)$ is an element of $\TTT(S)$ for every $\gamma<\kappa^+$. Let $C$ denote the set of limit points of $\ran{s}$ in $\kappa^+$. Then the definition of $\TTT(S)$ ensures that $C$ is a closed unbounded subset of $\kappa^+$ with $C\cap S^{\kappa^+}_\kappa\subseteq S$. In particular, it follows that $C$ witnesses that $S$ is an element of $\Club{S^{\kappa^+}_\kappa}$.  
 \end{proof}

 The next ingredient to connect the complexity of club filters to the existence of canary trees is the ordering of trees under \emph{order-embeddability}:

\begin{definition}
   Given trees $\SSS$ and $\TTT$, we let $\SSS\preceq\TTT$ denote the statement that there exists a function $\map{e}{\SSS}{\TTT}$ satisfying $e(s_0)<_\TTT e(s_1)$ for all $s_0,s_1\in\SSS$ with $s_0<_\SSS s_1$.  
\end{definition}

 The following result directly generalizes {\cite[Theorem 23]{MR1242054}} and its proof to successor cardinals of arbitrary infinite regular cardinals:

\begin{lemma}\label{lemma:CanaryEmb}
 Given an infinite regular cardinal $\kappa$ and a tree $\TTT$ of  height $\kappa^+$ with $[\TTT]=\emptyset$, consider the following statements: 
  \begin{enumerate}
    \item\label{item:CanaryChar2} $\TTT(S)\preceq\TTT$ holds for all subsets $S$ of $S^{\kappa^+}_\kappa$ that are bistationary in $S^{\kappa^+}_\kappa$.\footnote{Given a stationary subset $S$ of an uncountable regular cardinal $\theta$, we say that a subset $B$ of $S$ is \emph{bistationary in $S$} if both $B$ and $S\setminus B$ are stationary subsets of $\theta$.} 

    \item\label{item:CanaryChar1} $\TTT$ is a $\kappa$-canary tree. 
  \end{enumerate}
 Then \eqref{item:CanaryChar2} implies \eqref{item:CanaryChar1}. Moreover, if $\kappa^{{<}\kappa}=\kappa$ holds, then  \eqref{item:CanaryChar1} also implies \eqref{item:CanaryChar2}.  
\end{lemma}

\begin{proof}
 First, assume that \eqref{item:CanaryChar2} holds, let $S$ be a stationary subset of $S^{\kappa^+}_\kappa$ and let $\PPP$ be a ${<}\kappa^+$-distributive partial order with $\mathbbm{1}_\PPP\Vdash\anf{\check{S}\in\NS{\check{\kappa}^+}}$. Set $\SSS=\TTT(S^{\kappa^+}_\kappa\setminus S)$. 
   The fact that forcing with $\PPP$ preserves the regularity of $\kappa^+$ then ensures that  $S$ is bistationary in $S^{\kappa^+}_\kappa$. Our assumptions now yield a function $\map{e}{\SSS}{\TTT}$ with $e(s_0)<_\TTT e(s_1)$ for all $s_0,s_1\in\SSS$ with $s_0<_\SSS s_1$. 
   Now, let $G$ be $\PPP$-generic over $V$. Since $V$ and $V[G]$ contain the same bounded subsets of $\kappa^+$, we know that $\TTT(S^{\kappa^+}_\kappa\setminus S)^{V[G]}=\SSS$. Moreover, since $S$ is an element of $\NS{\kappa^+}^{V[G]}$, we know that $S^{\kappa^+}_\kappa\setminus S\in\Club{S^{\kappa^+}_\kappa}$ and Proposition \ref{proposition:TreesTS}.\ref{item:TreeStat2} shows that, in $V[G]$, the tree $\SSS$ has height $\kappa^+$ and contains a cofinal branch. But, this allows us to conclude that the tree $\TTT$ has a cofinal branch in $V[G]$.

 Now, assume that $\kappa^{{<}\kappa}=\kappa$ and \eqref{item:CanaryChar1} holds. 
 Fix a subset $S$ of $S^{\kappa^+}_\kappa$ that is  bistationary in $S^{\kappa^+}_\kappa$ and let $\PPP_S$ denote the canonical partial order to add a closed unbounded subset to $S\cup S^{\kappa^+}_{{<}\kappa}$, {i.e.,}  conditions in $\PPP_S$ are non-empty, closed and bounded subsets $c$ of $\kappa^+$ with $S^{\kappa^+}_\kappa\cap c\subseteq S$ and the ordering of $\PPP_S$ is given by reversed end-extension. 
 The assumption that $\kappa^{{<}\kappa}=\kappa$ holds then allows us to apply {\cite[Theorem 1]{MR0716625}} to show that $\PPP_S$ is ${<}\kappa^+$-distributive and  $\mathbbm{1}_{\PPP_S}\Vdash\anf{S^{\check{\kappa}^+}_{\check{\kappa}}\setminus\check{S}\in\NS{\check{\kappa}^+}}$ holds. 
  By our assumption, we can now find a $\PPP_S$-name $\dot{B}$ with $\mathbbm{1}_{\PPP_S}\Vdash\anf{\dot{B}\in[\check{\TTT}]}$. 
   In the following, we inductively construct a system $\seq{\langle t_s,c_s\rangle\in\TTT\times\PPP_S}{s\in\TTT(S)}$ such that the following statements hold for all $s\in\TTT(S)$: 
     
  \begin{enumerate}[label=(\roman*)]
   \item  $\max(\ran{s})\leq\length{t_s}{\TTT}\leq\max(c_s)$. 
   
   \item $c_s\Vdash_{\PPP_S}\anf{\check{t}_s\in\dot{B}}$. 
   
   \item If $r\in\TTT(S)$ with $r\leq_{\TTT(S)}s$, then $t_r\leq_\TTT t_s$ and $c_s\leq_{\PPP_S}c_r$. 
   
   \item If there is a $q\in\TTT(S)$ with $q<_{\TTT(S)}s$ and $\max(\ran{s})\leq\length{t_q}{\TTT}$, then $t_q=t_r$ and $c_q=c_r$ for all $r\in\TTT(S)$ with $q\leq_{\TTT(S)}r\leq_{\TTT(S)}s$.  
   
   \item If $\max(\ran{s})>\length{t_r}{\TTT}$  for all $r\in\TTT(S)$ with $r<_{\TTT(S)}s$, then $$\length{t_s}{\TTT} ~ > ~ \sup\Set{\max(c_r)}{r\in\TTT(S), ~ r<_{\TTT(S)}s}.$$ 
  \end{enumerate}

  First, if $s\in\TTT(S)$ with $\dom{s}=\{0\}$, then we define $t_s$ to be a minimal element of $\TTT$
 and set $c_s=\{0\}$. 
   Next, fix $s\in\TTT(S)$ with $\dom{s}>1$ and assume that we already defined pairs $\langle t_r,c_r\rangle$ with the above properties for all $r\in\TTT(S)$ with $r<_{\TTT(S)}s$. 
 If there is $r\in\TTT(S)$ with $r<_{\TTT(S)}s$ and $\max(\ran{s})\leq\length{t_r}{\TTT}$, then we define $t_s=t_r$ and $c_s=c_r$ and all of the above statements are satisfied.  
 Hence, we may assume that $\max(\ran{s})>\length{t_r}{\TTT}$ holds for all $r\in\TTT(S)$ with $r<_{\TTT(S)}s$.  
  First, assume that there is $r\in\TTT(S)$ with $r<_{\TTT(S)}s$ and $\dom{s}=\dom{r}+1$. Then we can easily find $t_s\in \TTT$ and $c_s\in\PPP_S$ such that $t_r<_\TTT t_s$, $\length{t_s}{\TTT}>\max(\ran{s})+\max(c_r)$, $\max(c_s)\geq\length{t_s}{\TTT}$ and $c_s\Vdash_{\PPP_S}\anf{\check{t}_s\in\dot{B}}$. 
 These choices then directly ensure that all of the above statements hold. 
 Now, assume that there is no maximal element below $s$ in $\TTT(S)$.

  \begin{claim*}
   If $\max(\dom{s})\in S^{\kappa^+}_\kappa$, then   $$\max(\ran{s}) ~ = ~ \sup\Set{\max(c_p)}{p\in\TTT(S), ~ p<_{\TTT(S)}s}.$$ 
 \end{claim*}
 
  \begin{proof}[Proof of the Claim]
  First, note that the assumption that $\max(\dom{s})\in S^{\kappa^+}_\kappa$ together with the fact that $s$ is continuous at all points of cofinality $\kappa$ in its domain imply  that 
   \begin{equation*}
    \begin{split}
      \max(\ran{s}) ~ & = ~ \sup\Set{\max(\ran{p})}{p\in\TTT(S), ~ p<_{\TTT(S)}s} \\ 
        &    \leq ~  \sup\Set{\max(c_p)}{p\in\TTT(S), ~ p<_{\TTT(S)}s}. 
      \end{split}
     \end{equation*}

    In the other direction, fix $p\in\TTT(S)$ with $p<_{\TTT(S)}s$. Then our assumptions imply that $\length{t_p}{\TTT}<\max(\ran{s})$ and the continuity of $s$ at $\max(\dom{s})$ allows us to find $r\in\TTT(S)$ that is $<_{\TTT(S)}$-minimal with the property that $r<_{\TTT(S)}s$ and $\length{t_p}{\TTT}<\max(\ran{r})$.
  Since $\max(\ran{p})\leq\length{t_p}{\TTT}$, we know that $p<_{\TTT(S)}r$.  
   Moreover, given $q\in\TTT(S)$ with $p\leq_{\TTT(S)}q<_{\TTT(S)}r$, we have  $\max(\ran{q})\leq\length{t_p}{\TTT}$ and this implies that $t_p=t_q$ and $c_p=c_q$.  
   This shows that $\max(\ran{r})>\length{t_q}{\TTT}$ holds for all $q\in\TTT(S)$ with $q<_{\TTT(S)}r$ and this allows us to  conclude that $$\max(c_p) ~ < ~ \length{t_r}{\TTT} ~ < ~ \max(\ran{s})$$ holds.  
 \end{proof}

 We   now  define $$c ~ = ~ \{\sup\Set{\max(c_p)}{p\in\TTT(S), ~ p<_{\TTT(S)}s}\} ~ \cup ~ \bigcup\Set{c_p}{p\in\TTT(S), ~ p<_{\TTT(S)}s}.$$ It is then easy to see that $c$ is a bounded and closed subset of $\kappa^+$. 
   Moreover, the above claim ensures that $S^{\kappa^+}_\kappa\cap c\subseteq S$ holds. 
   In particular, it follows that $c$ is a condition in $\PPP_S$ with $c\leq_{\PPP_S}c_p$ for all $p\in\TTT(S)$ with $p<_{\TTT(S)}s$.

  We can now pick $t_s\in\TTT$ and $c_s\in\PPP_S$ satisfying $$\length{t_s}{\TTT} ~ > ~ \sup\Set{\max(c_r)}{r\in\TTT(S), ~ r<_{\TTT(S)}s},$$ $t_r<_\TTT t_s$ for all $r\in\TTT(S)$ with $r<_{\TTT(S)}s$, $c_s\leq_{\PPP(S)}c$, $\length{t_s}{\TTT}\leq\max(c_s)$ and $c_s\Vdash_{\PPP_S}\anf{\check{t}_s\in\dot{B}}$. The pair $\langle t_s,c_s\rangle$ then satisfies the above statements. This completes our inductive construction.

 Finally, since $\length{s}{\TTT(S)}\leq\max(\ran{s})\leq\length{t_s}{\TTT}$ holds for all $s\in\TTT(S)$, there is a  function $\map{e}{\TTT(S)}{\TTT}$ with the property that for all $s\in\TTT(C)$, the set $e(s)$ is the unique element $t$ of $\TTT$ with $t\leq_\TTT t_s$ and $\length{s}{\TTT(S)}=\length{t}{\TTT}$. We can now conclude that the function $e$ witnesses that $\TTT(S)\preceq\TTT$ holds.  
\end{proof}

 We are now ready to relate the order-embeddability of trees to the complexity of restrictions of the club filter.

 \begin{lemma}\label{lemma:ComplexityFromTrees}
  Let $\kappa$ be an infinite regular cardinal, let $M$ be a stationary subset of $S^{\kappa^+}_\kappa$ and let $\TTT$ be a tree of height $\kappa^+$ with $[\TTT]=\emptyset$ and $\TTT(S^{\kappa^+}_\kappa\setminus S)\preceq\TTT$  for every subset $S$ of $M$ that is bistationary in $M$.  
  Then the set  $\NS{\kappa^+}\cap\POT{M}$ is definable by a $\Pi_1$-formula with parameters $\HH{\kappa^+}$, $M$, $S^{\kappa^+}_\kappa$ and $\TTT$. 
 \end{lemma}
 
 \begin{proof}
  Define $\calA$ to be the collection of all subsets $A$ of $M$ with the property   that either $A\in\Club{M}$  or $\TTT(S^{\kappa^+}_\kappa\setminus A)\preceq\TTT$ holds.

  \begin{claim*}
   $\calA=\POT{M}\setminus\NS{\kappa^+}$. 
  \end{claim*}
  
  \begin{proof}[Proof of the Claim]
   First, assume that $A$ is an element of $\NS{\kappa^+}\cap\POT{M}$. 
   Then $S^{\kappa^+}_\kappa\setminus A\in\Club{S^{\kappa^+}_\kappa}$ and the fact that  $M$ is a stationary subset of $S^{\kappa^+}_\kappa$ ensures that $A$ is not an element of $\Club{M}$.  
   Moreover, Proposition \ref{proposition:TreesTS}.\ref{item:TreeStat2} implies that $\TTT(S^{\kappa^+}_\kappa\setminus A)$ is a tree of height $\kappa^+$ with $[\TTT(S^{\kappa^+}_\kappa\setminus A)]\neq\emptyset$. By our assumptions on $\TTT$, this shows that $\TTT(S^{\kappa^+}_\kappa\setminus A)\preceq\TTT$ does not hold. We can now conclude that $A\notin\calA$.

    Now, assume that $A$ is a subset of $M$ that is not contained in $\NS{\kappa^+}$. If $A$ is an element of $\Club{M}$, then $A$ is an element of $\calA$. 
    Therefore, we may assume that $M\setminus A$ is stationary. In this situation, our assumptions imply that $\TTT(S^{\kappa^+}_\kappa\setminus A)\preceq\TTT$  and we know that $A\in\calA$.  
   \end{proof}

  The statement of the lemma now follows directly from the above claim, because   the set  $\calA$ is obviously definable by a $\Sigma_1$-formula with parameters $\HH{\kappa^+}$, $M$, $S^{\kappa^+}_\kappa$ and $\TTT$. 
 \end{proof}

 \begin{corollary}\label{corollary:CanaryFromDelta}
  Let $\kappa$ be an infinite regular cardinal with $2^\kappa=\kappa^+$. If there exists a tree of cardinality and height $\kappa^+$  with $[\TTT]=\emptyset$ and the property that $\TTT(S)\preceq\TTT$ holds for every subset $S$ of $S^{\kappa^+}_\kappa$ that is   bistationary in $S^{\kappa^+}_\kappa$, then $\Club{S^{\kappa^+}_\kappa}$ is a $\mathbf{\Delta}_1$-subset of $\POT{\kappa^+}$. 
  In particular, if $\kappa$ is an infinite cardinal with $\kappa^{{<}\kappa}=\kappa$ and $2^\kappa=\kappa^+$, then the existence of a $\kappa$-canary tree of cardinality $\kappa^+$ implies that $\Club{S^{\kappa^+}_\kappa}$ is a $\mathbf{\Delta}_1$-subset of $\POT{\kappa^+}$. 
\end{corollary}

\begin{proof}
 Without loss of generality, we may assume that $\TTT$ is an element of $\HH{\kappa^{++}}$. 
 An application of Lemma \ref{lemma:ComplexityFromTrees} with $M=S^{\kappa^+}_\kappa$ shows that the set $\NS{\kappa^+}\cap\POT{S^{\kappa^+}_\kappa}$ is definable by a $\Pi_1$-formula with parameters $\HH{\kappa^+}$, $M$, $S^{\kappa^+}_\kappa$ and $\TTT$. 
 Since a subset $A$ of $\kappa^+$ is an element of $\Club{S^{\kappa^+}_\kappa}$ if and only if $S^{\kappa^+}_\kappa\setminus A$ is  an element of $\NS{\kappa^+}\cap\POT{S^{\kappa^+}_\kappa}$, it follows that the set $\Club{S^{\kappa^+}_\kappa}$  is also definable by a $\Pi_1$-formula with parameters $\HH{\kappa^+}$, $M$, $S^{\kappa^+}_\kappa$ and $\TTT$. 
 Finally, our assumptions imply that all of these parameters are elements of $\HH{\kappa^{++}}$ and hence we can conclude that $\Club{S^{\kappa^+}_\kappa}$ is a $\mathbf{\Delta}_1$-subset of $\POT{\kappa}$.

  Now, assume that $\kappa$ is an infinite cardinal with $\kappa^{{<}\kappa}=\kappa$, $2^\kappa=\kappa^+$ and the property that there exists a $\kappa$-canary tree $\TTT$ of cardinality $\kappa^+$. Then $\kappa$ is regular and Lemma \ref{lemma:CanaryEmb} ensures that $\TTT(S)\preceq\TTT$ holds for all   subsets $S$ of $S^{\kappa^+}_\kappa$ that are bistationary in $S^{\kappa^+}_\kappa$. The first part of the corollary then allows us to conclude that $\Club{S^{\kappa^+}_\kappa}$ is a $\mathbf{\Delta}_1$-subset of $\POT{\kappa^+}$. 
\end{proof}

 \begin{remark}
   Using the \emph{Boundedness Lemma for uncountable regular cardinals} (see {\cite[Corollary 13]{MR1242054}} for $\kappa=\omega_1$ and {\cite[Lemma 8.1]{MR2987148}} for the direct generalization to higher regular cardinals), it is possible to show that the converse of the implication of Corollary \ref{corollary:CanaryFromDelta} also holds true, {i.e.,} if $\kappa$ is an infinite regular cardinal such that $2^\kappa=\kappa^+$ and $\Club{S^{\kappa^+}_\kappa}$ is a $\mathbf{\Delta}_1$-subset of $\POT{\kappa^+}$, then there is a tree $\TTT$ of cardinality and height $\kappa^+$ with the property that $\TTT(S)\preceq\TTT$ holds for every  subset $S$ of $S^{\kappa^+}_\kappa$ that is bistationary in $S^{\kappa^+}_\kappa$.  
\end{remark}

 The main results of \cite{MR1877015} and \cite{MR1222536} now show that, if $\kappa$ is an infinite cardinal satisfying $\kappa^{{<}\kappa}=\kappa$ and $2^\kappa=\kappa^+$, then there exists a cofinality-preserving  forcing that also preserves these cardinal arithmetic assumptions on $\kappa$ and adds a $\kappa$-canary tree of cardinality $\kappa^+$. By the second part of Corollary \ref{corollary:CanaryFromDelta}, this means that this forcing causes $\Club{S^{\kappa^+}_\kappa}$ to be a $\mathbf{\Delta}_1$-subset of $\POT{\kappa^+}$. 
  In the following, we will present work contained in \cite{MR4469807}   providing a detailed analysis of the forcing notion constructed in \cite{MR1877015} that leads to a strengthening of the main results of  \cite{MR1877015} and \cite{MR1222536}. 
  In particular, we will be able to relax the cardinal arithmetic assumptions on the given cardinal $\kappa$ and derive  strong closure properties of the constructed partial order.  This analysis is based on the following concept from Shelah's work on cardinal arithmetic:

 \begin{definition}[Shelah]
  Let $\kappa$ be an infinite regular cardinal. 
  \begin{enumerate}
   \item An ordinal $\gamma<\kappa^+$ is \emph{approachable} with respect to a sequence $\seq{z_\alpha}{\alpha<\kappa^+}$   of elements of $[\kappa^+]^{{<}\kappa}$  if there exists a cofinal sequence $\vec{\alpha} = \seq{\alpha_\xi}{\xi<\cof{\gamma}}$  in $\gamma$ such that every proper initial segment of $\vec{\alpha}$ is equal to $z_\alpha$ for some $\alpha<\gamma$.  
   
   \item The \emph{Approachability ideal $I[\kappa^+]$ on $\kappa^+$} is the (possibly non-proper) normal ideal generated by sets of the form $$A_{\vec{z}} ~ =  ~ \Set{\gamma<\kappa^+}{\textit{$\gamma$ is approachable with respect to $\vec{z}$ }}$$ for some sequence $\vec{z}\in{}^{\kappa^+}([\kappa^+]^{{<}\kappa})$. 
  \end{enumerate}
 \end{definition}

 \begin{lemma}[\cite{MR2160657}]\label{lemma:IA}
   Let $\kappa$ be an infinite regular cardinal with $\kappa^{{<}\kappa}\leq\kappa^+$, let $\seq{z_\alpha}{\alpha<\kappa^+}$ be an enumeration of $[\kappa^+]^{{<}\kappa}$ and set $$M_{\vec{z}} ~ =  ~ \Set{\gamma\in S^{\kappa^+}_\kappa}{\textit{$\gamma$ is approachable with respect to $\vec{z}$ }}.$$  Then the following statements hold: 
   \begin{enumerate}
     
     \item $M_{\vec{z}}\in I[\kappa^+]$. 

     \item $M_{\vec{z}}$ is a maximum element of $I[\kappa^+]\cap\POT{S^{\kappa^+}_\kappa} ~ mod ~  \mathcal{NS}$, in the sense that whenever $S$ is a stationary subset of $S^{\kappa^+}_\kappa$ with $S\in I[\kappa^+]$, then $S\setminus M_{\vec{z}}$ is non-stationary. 
     
     
     \item\label{item:IA3} If $\kappa^{{<}\kappa}=\kappa$, then $S^{\kappa^+}_\kappa\in I[\kappa^+]$. 
   \end{enumerate}
\end{lemma}

 Using the above notions and results, the work presented in \cite{MR4469807} leads to the following strengthening of the main results of \cite{MR1877015} and \cite{MR1222536}:

\begin{theorem}[\cite{MR4469807}]\label{theorem:CLcanary}
 Given an infinite regular cardinal $\kappa$, there is a partial order $\PPP$ with the following properties: 
 \begin{enumerate}
   \item $\PPP$ is ${<}\kappa^+$-directed closed and satisfies the $(2^\kappa)^+$-chain condition. 
   
   \item If $G$ is $\PPP$-generic over $V$, then, in $V[G]$, there is a subtree $\TTT$ of ${}^{{<}\kappa^+}\kappa^+$  of height $\kappa^+$ with $[\TTT]=\emptyset$ such that the following statements hold: 
      \begin{enumerate}
       \item If $S$ is a subset of $S^{\kappa^+}_\kappa$ that is bistationary in $S^{\kappa^+}_\kappa$ and the set  $S^{\kappa^+}_\kappa\setminus S$ contains a stationary set in $I[\kappa^+]$, then $\TTT(S)\preceq\TTT$. 
             
       \item  If  $M\in V$ is a maximum element of $I[\kappa^+]\cap\POT{S^{\kappa^+}_\kappa} ~ mod ~ \mathcal{NS}$ in $V$ and $\kappa^{{<}\kappa}\leq\kappa^+$ holds in $V$, then the following statements hold in $V[G]$: 
         \begin{enumerate}
           \item $M$ is a maximum element of $I[\kappa^+]\cap\POT{S^{\kappa^+}_\kappa} ~ mod ~ \mathcal{NS}$. 
           
           \item If $S$ is a subset of $S^{\kappa^+}_\kappa$ that is bistationary in $S^{\kappa^+}_\kappa$ and $M\setminus S$ is stationary, then  $\TTT(S)\preceq\TTT$.  
         \end{enumerate}
      \end{enumerate}
 \end{enumerate}
\end{theorem}

The conclusions of the above theorem enable  us to give a detailed analysis of the complexities of various restrictions  of the club filter  in the constructed forcing extension:

\begin{corollary}\label{corollary:ComplexityForcingNS}
  Let $\kappa$ be an infinite regular cardinal satisfying $\kappa^{{<}\kappa}\leq\kappa^+$, let $\PPP$ be the partial order given by Theorem \ref{theorem:CLcanary} and let $M$ be a maximum element of $I[\kappa^+]\cap\POT{S^{\kappa^+}_\kappa} ~ mod ~  \mathcal{NS}$. If $G$ is $\PPP$-generic over $V$, then, in $V[G]$, the set $\NS{\kappa^+}\cap\POT{M}$ is definable by a $\Pi_1$-formula with parameters in $\HH{(2^\kappa)^+}$. 
\end{corollary}

\begin{proof}
   Let $\TTT\in V[G]$ be the tree given by Theorem \ref{theorem:CLcanary}. Then $\TTT\in\HH{(2^\kappa)^+}^{V[G]}$. 
   Work in $V[G]$ and fix  $S\subseteq M$  bistationary in $M$.  Then $M\setminus(S^{\kappa^+}_\kappa\setminus S)$ is stationary and $S^{\kappa^+}_\kappa\setminus S$ is  bistationary in $S^{\kappa^+}_\kappa$. Hence, Theorem \ref{theorem:CLcanary} implies $\TTT(S^{\kappa^+}_\kappa\setminus S)\preceq\TTT$. An application of Lemma \ref{lemma:ComplexityFromTrees} now shows that $\NS{\kappa^+}\cap\POT{M}$ is definable by a $\Pi_1$-formula with parameters $\HH{\kappa^+}$, $M$, $S^{\kappa^+}_\kappa$ and $\TTT$. This conclusion now directly implies the statement of the lemma, because all of these parameters are contained in $\HH{(2^\kappa)^+}$.  
\end{proof}

We now continue by showing how the main results of  \cite{MR1877015} and \cite{MR1222536} can be directly derived from the statement of Theorem \ref{theorem:CLcanary}.

\begin{theorem}[\cite{MR1877015,MR1222536}]
 Let $\kappa$ be an infinite cardinal with $\kappa^{{<}\kappa}=\kappa$ and $2^\kappa=\kappa^+$, let $\PPP$ be the partial order given by Theorem \ref{theorem:CLcanary} and let $G$ be $\PPP$-generic over $V$. 
 Then, in $V[G]$, there is a $\kappa$-canary tree of cardinality $\kappa^+$ and $\Club{S^{\kappa^+}_\kappa}$ is a $\mathbf{\Delta}_1$-subset of $\POT{\kappa}$. 
\end{theorem}

\begin{proof}
  Set $M=(S^{\kappa^+}_\kappa)^V$. Then Lemma \ref{lemma:IA}.\ref{item:IA3} implies that $M$ is a maximum element of $I[\kappa^+]\cap\POT{S^{\kappa^+}_\kappa} ~ mod ~ \mathcal{NS}$ in $V$. 
   Let $\TTT\in V[G]$ be the tree given by Theorem \ref{theorem:CLcanary}. We then know that, in $V[G]$, the tree $\TTT$ has height and cardinality $\kappa^+$ and $[\TTT]=\emptyset$ holds. 
   Moreover, if $S\subseteq S^{\kappa^+}_\kappa$ is bistationary in $S^{\kappa^+}_\kappa$ in $V[G]$, then the last item in the statement  of Theorem \ref{theorem:CLcanary} implies that $\TTT(S)\preceq\TTT$ holds. 
 Since the properties of $\PPP$ ensure that $2^\kappa=\kappa^+$ holds in $V[G]$, we can now apply Corollary \ref{corollary:CanaryFromDelta} to show that  $\Club{S^{\kappa^+}_\kappa}$ is a $\mathbf{\Delta}_1$-subset of $\POT{\kappa}$. 
 Finally, the fact that $\kappa^{{<}\kappa}=\kappa$ holds in $V[G]$ allows us to use Lemma \ref{lemma:CanaryEmb} to conclude that $\TTT$ is a $\kappa$-canary tree in $V[G]$.  
\end{proof}

 In the remainder of this section, we discuss the main result of \cite{MR4469807} that uses Theorem \ref{theorem:CLcanary} to show that strong forcing axioms are compatible with the $\mathbf{\Delta}_1$-definability of $\Club{S^2_0}$. 
  The proof of this result relies on a connection between principles of \emph{stationary reflection} and the complexities  of non-stationary ideals that we will discuss next. 
  Given an uncountable regular cardinal $\theta$ and a stationary subset $S$ of $\theta$, we let $\Refl{S}$ denote the set of all \emph{reflection points of $S$ in $\theta$}, {i.e.,} the set of all limit ordinals $\lambda<\theta$ with the property that $S\cap\lambda$ is stationary in $\lambda$. 

\begin{lemma}\label{lemma:DefFromRefl}
 Let $S$ be a stationary subset of an uncountable regular cardinal $\theta$ and let $\calE$  be a set of stationary subsets of $\theta$ with the property that for every stationary subset $A$ of $S$, there exists $E\in\calE$ with $E\subseteq\Refl{A}$. 
 If the set $\calE$ is definable by a $\Sigma_1$-formula with parameter $p$, then the set $\NS{\theta}\cap\POT{S}$ is definable by a $\Pi_1$-formula with parameters $\HH{\theta}$, $p$ and $S$. 
\end{lemma}
  
\begin{proof}
  We define  $\calA$ to be the set of all subsets $A$ of $S$ with the property that there exists $E\in\calE$ with $E\subseteq\Refl{A}$. 
   If $A\in\POT{S}\setminus\NS{\theta}$, then our assumptions on $\calE$ ensure that $A$ is contained in $\calA$. 
   Moreover, if $E\in\calE$ witnesses that $A\subseteq S$ is an element of $\calA$ and $C$ is  closed unbounded in  $\theta$, then we can find $\lambda\in E\cap\Lim(C)\subseteq\Refl{A}$ and therefore $\emptyset\neq A\cap C\cap\lambda\subseteq A\cap C$. 
   In combination, these computations show that $\calA=\POT{S}\setminus\NS{\theta}$ and this yields the conclusion of the lemma, because the set $\calA$  is  definable by a $\Sigma_1$-formula with parameters $\HH{\theta}$, $p$ and $S$. 
\end{proof}

 Before we connect the above lemma with Theorem \ref{theorem:CLcanary}, we need to recall another concept that is closely connected to the notions discussed above. 
 Remember that, given   an uncountable regular  cardinal $\kappa$, the class  $\text{IA}_\kappa$ of all sets that are \emph{internally approachable with length and cardinality $\kappa$} consists of all sets $W$ with the property that there exists a sequence $\vec{N} = \seq{N_\alpha}{\alpha < \kappa}$  satisfying the following statements: 
 \begin{enumerate}
   \item The sequence $\vec{N}$ is $\subseteq$-increasing and $\subseteq$-continuous. 
  
  \item $W = \bigcup\Set{N_\alpha}{\alpha < \kappa}$. 
 
  \item $\vert N_\alpha\vert < \kappa$ for all $\alpha<\kappa$. 
 
  \item Every proper initial segment of $\vec{N}$ is an element of $W$. 
 \end{enumerate}

\begin{theorem}[\cite{MR4469807}]
  Assume that \emph{Martin's Maximum $\MM$} holds and let $\PPP$ denote the partial order given by Theorem \ref{theorem:CLcanary}. 
 If $G$ is $\PPP$-generic over $V$, then $\Club{S^2_0}$ is a $\mathbf{\Delta}_1$-subset of $\POT{\omega_2}$ in $V[G]$. 
\end{theorem}

\begin{proof}
  First, note that our assumptions imply that $2^{\aleph_0}=2^{\aleph_1}=\aleph_2$ holds in $V$ (see {\cite[Theorem 16.20 \& 31.23]{MR1940513}}) and $\MM$ holds in $V[G]$ (see {\cite[Theorem 4.7]{cox2018forcing}}).
  Lemma \ref{lemma:IA} then shows that, in $V$, there is a subset $M$ of $S^2_1$ that is a maximum element of $I[\omega_2]\cap\POT{S^2_1} ~ mod ~  \mathcal{NS}$ and Theorem \ref{theorem:CLcanary} shows that $M$ retains this property in $V[G]$. 
  In addition, an application of Corollary \ref{corollary:ComplexityForcingNS} allows us to conclude that, in $V[G]$, the set $\NS{\omega_2}\cap\POT{M}$ is definable by a $\Pi_1$-formula with parameters in $\HH{\aleph_3}$.  
  Now, work in $V[G]$ and define $\calE=\POT{M}\setminus\NS{\omega_2}$. We then know that the set $\calE$ consists of stationary subsets of $\omega_2$ and it is definable by a $\Sigma_1$-formula with parameters in $\HH{\aleph_3}$.

  \begin{claim*}
     If $A$ is a stationary subset of $S^2_0$, then  there is $E\in\calE$ with $E\subseteq\Refl{A}$. 
  \end{claim*}
  
  \begin{proof}[Proof of the Claim]
   The proof of {\cite[Theorem 13]{MM}}  yields $R\subseteq\text{IA}_{\omega_1}\cap[\HH{\aleph_3}]^{\aleph_1}$ that is stationary in $[\HH{\aleph_3}]^{\aleph_1}$, consists of elementary substructures of $\HH{\aleph_3}$ and satisfies  $E_0=\Set{W\cap\omega_2}{W\in R}\subseteq\Refl{A}$. 
    Then $E_0$ is a stationary subset of $S^2_1$. 
    Moreover, the fact that $2^{\aleph_0}=\aleph_2$ holds implies that each $W$ in $R$ contains an enumeration $\vec{z}=\seq{z_\xi}{\xi<\omega_2}$ of $[\omega_2]^{\aleph_0}$ as an element and therefore  the internal approachability of $W$ implies that $E_0$ is approachable with respect to $\vec{z}$. 
    This shows that $E_0$ is a stationary element of $I[\omega_2]$. 
    Since $M$ is a maximum element of $I[\omega_2]\cap\POT{S^2_1} ~ mod ~  \mathcal{NS}$,  we now know that $E=E_0\cap M$ is a stationary subset of $M$. In particular, we can conclude that  $E$ is an element of $\calE$ with $E\subseteq\Refl{A}$.  
  \end{proof}
  
  The above claim now allows us to use Lemma \ref{lemma:DefFromRefl} to show that $\NS{\omega_2}\cap\POT{S^2_0}$ is a $\mathbf{\Delta}_1$-subset of $\POT{\omega_2}$. 
  Since $$\Club{S^2_0} ~ = ~ \Set{A\subseteq\kappa}{S^2_0\setminus A\in\NS{\kappa^+}\cap\POT{S^2_0}},$$ this also shows that $\Club{S^2_0}$ is a $\mathbf{\Delta}_1$-subset of $\POT{\omega_2}$.  
\end{proof}

 It should be noted that the results of \cite{MR4469807} also show that strong forcing axioms are compatible with a negative answer to Question \ref{question:Q2} at $\omega_2$.
  This is a direct consequence of the indestructibility of these axioms under ${<}\omega_2$-directed closed forcings and   Lemma \ref{lemma:AddDeltaAbsoDelta} below.


\subsection{Stationary reflection}
 In this short section, we present an observation from  \cite{MR4469807} that shows that the connection between stationary reflection and the complexity of the non-stationary ideal provided by Lemma \ref{lemma:DefFromRefl} can be further utilized to obtain affirmative answers to Question \ref{question:Q2}. 
  Note that results presented later in this paper show that the assumptions of the following theorem cannot be reduced to the statement that every stationary subset of $S$ reflects (see Corollary \ref{corollary:NonDeltaWC} below).

\begin{proposition}
 Let $E$ and $S$ be stationary subsets of an uncountable cardinal with $\kappa=\kappa^{{<}\kappa}$. If  $\Refl{A}\in\Club{E}$ holds for every stationary subset $A$ of $S$, then $\Club{S}$ is a $\mathbf{\Delta}_1$-subset of $\POT{\kappa}$.
\end{proposition}

\begin{proof}
 First, note that $\Club{E}$ is a set of stationary subsets of $\kappa$ that is definable by a $\Sigma_1$-formula with parameter $E$ and has the property that for every stationary subset $A$ of $S$, there exists $D\in\Club{E}$ with $D\subseteq\Refl{A}$. 
  Hence, Lemma \ref{lemma:DefFromRefl}  shows that the set $\NS{\kappa}\cap\POT{S}$ is definable by a $\Pi_1$-formula with parameters $E$, $\HH{\kappa}$ and $S$.  
  Since $\HH{\kappa}$ is an element of $\HH{\kappa^+}$ and $$\Club{S} ~ = ~ \Set{A\subseteq\kappa}{S\setminus A\in\NS{\kappa}\cap\POT{S}},$$ we can conclude that $\Club{S}$ is a $\mathbf{\Delta}_1$-subset of $\POT{\kappa}$.   
\end{proof}

 In \cite{MR683153}, Magidor shows that it is possible to force over a model of $\ZFC+\GCH$ containing a weakly compact cardinal to produce a model of $\ZFC+{2^{\aleph_1}=\aleph_2}$ with the property that $\Refl{A}\in\Club{S^2_1}$ holds for every stationary subset $A$ of $S^2_0$. 
  The above proposition now shows that $\Club{S^2_0}$ is a $\mathbf{\Delta}_1$-subset of $\POT{\omega_2}$ in Magidor's model.


\subsection{Another forcing result}

 As the final definability result, we present the statement of the main result of \cite{MR3320477}. 
 While the  consistency results discussed above yield several settings in which there is an infinite regular cardinal $\kappa$ and a stationary subset $S$ of $\kappa^+$ with the property that the restriction $\Club{S}$ of the club filter  to $S$  is a $\mathbf{\Delta}_1$-subset of $\POT{\kappa^+}$, this theorem of S. Friedman, Wu and Zdomskyy shows that the full club filter $\Club{\kappa^+}$ of an arbitrary successor cardinal $\kappa^+$ can be forced to be a $\mathbf{\Delta}_1$-subset of $\POT{\kappa^+}$.

 \begin{theorem}[\cite{MR3320477}]
  Assume that $V=L$ holds. If  $\kappa$ is an infinite cardinal, then there exists a partial order $\PPP$ such that the following statements hold: 
  \begin{enumerate}
   \item Forcing with $\PPP$ preserves all cardinals and the $\GCH$. 
   
   \item If $G$ is $\PPP$-generic over $V$, then $\Club{\kappa^+}$ is a $\mathbf{\Delta}_1$-subset of $\POT{\kappa^+}$ in $V[G]$. 
  \end{enumerate}
 \end{theorem}


\section{Undefinability results}\label{section:3}

 We now continue by presenting several canonical settings in which  non-stationary ideals are not $\mathbf{\Delta}_1$-definable.


\subsection{Descriptive arguments}\label{subsection:Descriptive}

The easiest way to obtain a negative answer to Question \ref{question:Q1} at an uncountable regular cardinal  is to start with an uncountable cardinal $\kappa$ satisfying  $\kappa^{{<}\kappa}=\kappa$  and then add $\kappa^+$-many Cohen subsets to $\kappa$  (see {\cite[Theorem 4.1]{MR1877015}} and {\cite[Theorem 3]{MR1222536}}).  
 In the following, we want to present results contained in \cite{MR3952233} and \cite{MR3430247}  that put this consistency result into a wider context. 
 More specifically, we want to derive this implication using notions and results from  \emph{generalized descriptive set theory}, the study of the structural properties of definable subsets of higher higher function spaces (see, for example, \cite{MR3235820}, \cite{MR3549563}, \cite{MR2987148} and \cite{MR1242054}).  
 For this purpose, we equip the power set $\POT{\kappa}$ of an infinite regular cardinal $\kappa$ with the topology whose basic open sets are of the form $$\calN_{\beta,b} ~ = ~ \Set{x\subseteq\kappa}{x\cap\beta=b}$$ for some ordinal $\beta<\kappa$ and a subset $b$ of $\beta$. 
 The resulting space can easily be identified with the \emph{generalized Cantor space} of the cardinal $\kappa$, {i.e.,} the topological space consisting of the set ${}^\kappa 2$ of all functions from $\kappa$ to $\{0,1\}$ equipped with the topology whose basic open sets consist of all extensions of functions $\map{s}{\alpha}{2}$ with $\alpha<\kappa$. 
 The next definition directly generalizes the classical Baire property to higher spaces.

 \begin{definition}
  Let $\calX$ be a topological space and let $\kappa$ be an infinite cardinal. 
  \begin{enumerate}
   \item A subset of $\calX$ is \emph{nowhere dense}  if its closure in $\calX$ has empty interior. 
   
   \item A subset of $\calX$ is \emph{$\kappa$-meager} if it is equal to the union of $\kappa$-many nowhere dense subsets of $\calX$. 
   
   \item A subset $\calA$ of $\calX$ has the \emph{$\kappa$-Baire property} if there is an open subset $\calU$ of $\calX$ with the property that the symmetric difference $$\calA_\Delta\calU ~ = ~ (\calA\setminus\calU) ~ \cup ~ (\calU\setminus\calA)$$ is $\kappa$-meager. 
  \end{enumerate}
 \end{definition}

In \cite{MR1880900}, Halko and Shelah showed that, if $\kappa$ is an uncountable regular cardinal, then $\Club{\kappa}$ and $\NS{\kappa}$ are subsets of $\POT{\kappa}$ without the $\kappa$-Baire property. In the following, we will present an argument contained in  \cite{MR3430247} that yields a stronger conclusion.  This argument is based on the following topological property:

\begin{definition} 
 Given an infinite cardinal $\kappa$, a subset $\calS$ of a topological space $\calX$ is \emph{$\kappa$-super-dense}  if for every non-empty  open subset $\calU$ of $\calX$ and every sequence $\seq{\calU_\alpha}{\alpha<\kappa}$ of dense open subsets of $\calU$, the corresponding  intersection $\bigcap\Set{\calS\cap\calU_\alpha}{\alpha<\kappa})$ is non-empty. 
\end{definition}

 The following proposition motivates the definition of super-denseness. It generalizes the trivial fact that disjoint dense subsets of a topological space cannot be separated by an open subset.

\begin{proposition} 
 Let $\kappa$ be an infinite cardinal and let $\calS$ and $\calT$ be disjoint $\kappa$-super-dense subsets of a topological space $\calX$. 
 If $\calA$ is a subset of $\calX$ that separates $\calS$ from $\calT$({i.e.,} $\calS\subseteq\calA\subseteq\calX\setminus\calT$ holds), then  $\calA$ does not have the $\kappa$-Baire property.    
\end{proposition}

\begin{proof}
 Assume, towards a contradiction, that there is an open subset $\calU$ of $\calX$ and a sequence 
 $\seq{\calC_\alpha}{\alpha<\kappa}$ of closed nowhere dense subsets of $\calX$ with the property that  $\calA_\Delta \calU  \subseteq  \bigcup\Set{\calC_\alpha}{\alpha<\kappa}$. 
 First, assume, towards a contradiction, that $\calU$ is the empty set and therefore $\calS\subseteq\calA  \subseteq  \bigcup\Set{\calC_\alpha}{\alpha<\kappa}$. 
 Then $\seq{\calX\setminus\calC_\alpha}{\alpha<\kappa}$ is a sequence of dense open subsets of $\calX$, and the $\kappa$-super-density of $\calS$ implies that $\bigcap\Set{\calS\setminus\calC_\alpha}{\alpha<\kappa}\neq\emptyset$, a contradiction.  
 This shows that $\calU$ is non-empty and $\seq{\calU\setminus\calC_\alpha}{\alpha<\kappa}$ is a sequence of dense open subsets of $\calU$. 
 Then there is an $x\in\bigcap\Set{(\calT\cap\calU)\setminus\calC_\alpha}{\alpha<\kappa}$. 
 Since $x\in\calT$, we have $x\in\calU\setminus\calA$ and therefore $x\in\calC_\alpha$ for some $\alpha<\kappa$, a contradiction.  
\end{proof}

The following argument now shows how  super-denseness is connected to our context:

\begin{lemma}\label{00:lemma:ClubFilterSuperDense}
 If $\kappa$ is an uncountable regular cardinal, then $\Club{\kappa}$ and $\NS{\kappa}$ are disjoint $\kappa$-super-dense subsets of $\POT{\kappa}$. 
\end{lemma}

\begin{proof}
 Let $\calU$ be a non-empty open subset of $\POT{\kappa}$ and let $\seq{\calU_\alpha}{\alpha<\kappa}$ be a sequence of dense open subsets of $\calU$. Pick $\beta<\kappa$ and $b\subseteq\beta$ with $\calN_{\beta,b}\subseteq\calU$. 
 In this situation, a standard inductive construction yields a strictly increasing, continuous sequence $\seq{\beta_\alpha<\kappa}{\alpha<\kappa}$ of ordinals and a system $\seq{b^i_\alpha}{\alpha<\kappa, ~ i<2}$ of bounded subsets of $\kappa$  with the property that 
 $b^i_\alpha\subseteq\beta_\alpha$, 
  $b=b^i_\alpha\cap\beta$, 
  $b^i_\alpha=b^i_{\bar{\alpha}}\cap\beta_\alpha$, $\beta_\alpha\notin b^0_{\bar{\alpha}}$, $\beta_\alpha\in b^1_{\bar{\alpha}}$ and
  $\calN_{\beta_\alpha,b^i_\alpha} ~ \subseteq ~ \calU_\alpha$  
   for all $\alpha<\bar{\alpha}<\kappa$ and $i<2$. 
  Set $C=\Set{\beta_\alpha}{\alpha<\kappa}$ and $B_i=\bigcup\Set{b^i_\alpha}{\alpha<\kappa}\subseteq\kappa$ for $i<2$. 
  Then $C$ is a closed unbounded subset of $\kappa$ that witnesses that $B_0\in\Club{\kappa}$ and $B_1\in\NS{\kappa}$. Moreover, our construction ensures that $B_0,B_1\in\bigcap\Set{\calU_\alpha}{\alpha<\kappa}$. 
\end{proof}

 \begin{corollary}
  If $\kappa$ is an uncountable regular cardinal, then no subset of $\POT{\kappa}$ that separates $\Club{\kappa}$ from $\NS{\kappa}$ has the $\kappa$-Baire property. \qed 
 \end{corollary}

 In combination with results presented in Section \ref{section:2}, this corollary shows that, if  $\ZFC$ is consistent, then this theory does not prove that for uncountable regular cardinals $\kappa$, all $\mathbf{\Delta}_1$-subsets of $\POT{\kappa}$ have the $\kappa$-Baire property. 
 In the remainder of this section, we will show that the negation of this statement is also not provable. 
 These arguments connect the regularity property introduced above with notions of \emph{generic absoluteness}.

\begin{definition}
 Given an infinite cardinal $\kappa$ and a partial order $\PPP$, a subset $\calA$ of $\POT{\kappa}$ is \emph{$\PPP$-absolutely $\mathbf{\Delta}_1$-definable}  if there is $z\in\HH{\kappa^+}$ and  $\Sigma_1$-formulas $\varphi_0(v_0,v_1)$ and $\varphi_1(v_0,v_1)$ such that the following statements hold: 
 \begin{itemize}
  \item  $\calA=\Set{A\subseteq\kappa}{\varphi_0(A,z)}$. 
  
  \item $\POT{\kappa}\setminus\calA=\Set{A\subseteq\kappa}{\varphi_1(A,z)}$. 
  
  \item $\mathbbm{1}_\PPP\Vdash\anf{\forall A\subseteq\check{\kappa} ~ [\varphi_0(A,\check{z})\vee\varphi_1(A,\check{z})]}$. 
 \end{itemize} 
\end{definition}
 
 In the following, we let  $\Add{\kappa}{1}$ denote  the forcing that adds a Cohen subset to an infinite regular cardinal $\kappa$, {i.e.,} the partial order that consists of functions $\map{s}{\alpha}{2}$ with $\alpha<\kappa$, ordered under end-extension.

\begin{lemma}\label{lemma:AbsoBaireProperty}
 If $\kappa$ is an infinite cardinal with $\kappa=\kappa^{{<}\kappa}$, then every  $\Add{\kappa}{1}$-absolutely $\mathbf{\Delta}_1$-definable subset of $\POT{\kappa}$ has the $\kappa$-Baire property. 
\end{lemma}

\begin{proof}
  Fix $z\in\HH{\kappa^+}$ and  $\Sigma_1$-formulas $\varphi_0(v_0,v_1)$ and $\varphi_1(v_0,v_1)$ witnessing that a subset $\calA$ of $\POT{\kappa}$ is $\Add{\kappa}{1}$-absolutely $\mathbf{\Delta}_1$-definable. Then 
  \begin{equation}\label{equation:ForcingDelta}
 \mathbbm{1}_{\Add{\kappa}{1}}\Vdash\anf{\forall A\subseteq\check{\kappa} ~ [\varphi_0(A,\check{z})\vee\varphi_1(A,\check{z})]}.
\end{equation}  
 holds. 
 Given an $\Add{\kappa}{1}$-name $\dot{A}$ for a subset of $\kappa$ with  $\dot{A}\in\HH{\kappa^+}$, our assumptions ensure that the statement $$\mathbbm{1}_{\Add{\kappa}{1}}\Vdash\anf{\varphi_0(\dot{A},\check{z})\vee\varphi_1(\dot{A},\check{z})}$$  can be formulated by a $\Sigma_1$-formula with parameters in $\HH{\kappa^+}$ (see {\cite[SectionVII.3]{MR597342}}) and therefore $\Sigma_1$-absoluteness ensures that this statement holds in $\HH{\kappa^+}$.  In particular, we know that \eqref{equation:ForcingDelta} holds in $\HH{\kappa^+}$.

  Let $M$ be an elementary submodel of $\HH{\kappa^+}$ of cardinality $\kappa$ with the property that $\kappa^{{<}\kappa}\cup\{z\}\subseteq M$. Then $M$ is transitive and   elementarity implies that \eqref{equation:ForcingDelta} holds in $M$.
   Let $\seq{\calC_\alpha}{\alpha<\kappa}$ enumerate all closed nowhere dense subsets of $\POT{\kappa}$ that are elements of $M$ and  let $\dot{A}$ denote the canonical $\Add{\kappa}{1}$-name for the generic subset of $\kappa$, {i.e.,} we have $\dot{A}^G=\Set{\alpha<\kappa}{(\bigcup G)(\alpha)=1}$ whenever $G$ is $\Add{\kappa}{1}$-generic over $V$. Given a condition $s$ in $\Add{\kappa}{1}$, we set $\beta_s=\dom{s}$ and $b_s=\Set{\alpha<\beta_s}{s(\alpha)=1}$. 
   Now, define $$\calU ~ = ~ \bigcup\Set{\calN_{\beta_s,b_s}}{s\Vdash^M_{\Add{\kappa}{1}}\varphi_0(\dot{A},\check{z})} ~ \subseteq ~ \POT{\kappa}.$$

 \begin{claim*}
  $\calA\setminus\calU \subseteq\bigcup\Set{\calC_\alpha}{\alpha<\kappa}$. 
 \end{claim*}

 \begin{proof}[Proof of the Claim]
  Pick $A\in\calA\setminus\bigcup\Set{\calC_\alpha}{\alpha<\kappa}$ and let $\map{x}{\kappa}{2}$ denote the characteristic function of $A$.  
  Since $A$ is an element of every dense open subset of $\POT{\kappa}$ that is an element of $M$, 
  it follows that the filter $G_A=\Set{x\restriction\alpha}{\alpha<\kappa}$ is $\Add{\kappa}{1}$-generic over $M$ with $A=\dot{A}^{G_A}\in M[G_A]$. 
    Moreover, since $A\in\calA$, we know that $\varphi_1(A,z)$ does not hold in $V$ and hence $\Sigma_1$-upwards absoluteness implies that  $\varphi_1(A,z)$ fails in $M[G_A]$. 
    The fact that \eqref{equation:ForcingDelta} holds in $M$ then ensures that $\varphi_0(A,z)$ holds in $M[G_A]$ and therefore we can find  $\alpha<\kappa$ such that   $$x{\restriction}\alpha ~ \Vdash^M_{\Add{\kappa}{1}} ~ \varphi_0(\dot{A},\check{z})$$ holds.  This allows us to conclude that $A\in\calN_{\beta_{x\restriction\alpha},b_{x\restriction\alpha}}\subseteq\calU$. 
 \end{proof}

 \begin{claim*}
  $\calU\setminus\calA\subseteq\bigcup\Set{\calC_\alpha}{\alpha<\kappa}$. 
 \end{claim*}

 \begin{proof}[Proof of the Claim]
  Pick $A\in\calU\setminus\bigcup\Set{\calC_\alpha}{\alpha<\kappa}$ and let $\map{x}{\kappa}{2}$ denote the corresponding characteristic function. As above, we know that the  filter $G_A=\Set{x\restriction\alpha}{\alpha<\kappa}$ is $\Add{\kappa}{1}$-generic over $M$ with $A=\dot{A}^{G_A}\in M[G_A]$. Since $A$ is an element of $\calU$, it follows that $\varphi_0(A,z)$ holds in $M[G_A]$ and therefore $\Sigma_1$-upwards absoluteness implies that $A\in\calA$. 
 \end{proof}

 In combination, the above claims show that the open subset $\calU$ and the sequence $\seq{\calC_\alpha}{\alpha<\kappa}$ witness that $\calA$ has the $\kappa$-Baire property.  
\end{proof}

Our next aim is to show that, consistently, all $\mathbf{\Delta}_1$-subsets of $\POT{\kappa}$ can be   $\Add{\kappa}{1}$-absolutely $\mathbf{\Delta}_1$-definable. For the proof of this consistency result, we need a well-known observation that goes back to the work of Silver.

\begin{lemma}\label{lemma:Silver}
 If $\kappa$ is an infinite regular cardinal, $\PPP$ is a ${<}\kappa$-closed partial order, $\varphi(v_0,\ldots,v_n)$ is a $\Sigma_0$-formula, $\dot{A}$ is a $\PPP$-name for a subset of $\kappa$ and $B_0,\ldots,B_{n-1}\subseteq\kappa$ with $\mathbbm{1}_\PPP\Vdash\varphi(\dot{A},\check{B}_0,\ldots,\check{B}_{n-1})$, then there is $A\subseteq\kappa$ with $\varphi(A,B_0,\ldots,B_{n-1})$. 
\end{lemma}

\begin{proof}
 By a standard elementary submodel argument, there is a $\PPP$-name $\dot{C}$ for a closed unbounded subset of $\kappa$ with the property that whenever $G$ is $\PPP$-generic over $V$ and $\alpha\in\dot{C}^G$, then $\varphi(\dot{A}^G\cap\alpha,B_0\cap\alpha,\ldots,B_{n-1}\cap\alpha)$ holds. 
 We can then construct a descending sequence $\seq{p_\alpha}{\alpha<\kappa}$ of conditions in $\PPP$, a strictly increasing, continuous sequence $\seq{\beta_\alpha}{\alpha<\kappa}$ of ordinals less than $\kappa$ and a sequence $\seq{a_\alpha}{\alpha<\kappa}$ of bounded subsets of $\kappa$ such that $$p_\alpha\Vdash_\PPP\anf{\check{\beta}_\alpha\in\dot{C}\wedge\dot{A}\cap\check{\beta}_\alpha=\check{a}_\alpha}$$ holds for all $\alpha<\kappa$. 
 Set $A=\bigcup\Set{a_\alpha}{\alpha<\kappa}\subseteq\kappa$ and assume, towards a contradiction, that $\varphi(A,B_0,\ldots,B_{n-1})$ does not hold. Another elementary submodel argument then yields an $\alpha<\kappa$ such that $\varphi(A\cap\beta_\alpha,B_0\cap\beta_\alpha,\ldots,B_{n-1}\cap\beta_\alpha)$ fails. 
 Let $G$ be $\PPP$-generic over $V$ with $p_\alpha\in G$. Then $A\cap\beta_\alpha=\dot{A}^G\cap\beta_\alpha$ and $\beta_\alpha\in\dot{C}^G$. But, this implies that $\varphi(A\cap\beta_\alpha,B_0\cap\beta_\alpha,\ldots,B_{n-1}\cap\beta_\alpha)$ holds in $V[G]$, contradicting the fact that $\Sigma_0$-formulas with parameters in $V$ are absolute between $V$ and $V[G]$. 
\end{proof}

We now show that the forcing $\Add{\kappa}{\kappa^+}$ that adds $\kappa^+$-many Cohen subsets to an infinite regular cardinal $\kappa$ allows us to construct the desired model of set theory.

\begin{lemma}\label{lemma:AddDeltaAbsoDelta}
 Let $\kappa$ be an infinite cardinal with  $\kappa=\kappa^{{<}\kappa}$ and let $G$ be $\Add{\kappa}{\kappa^+}$-generic over $V$. In $V[G]$, every $\mathbf{\Delta}_1$-subset of $\POT{\kappa}$ is $\Add{\kappa}{1}$-absolutely $\mathbf{\Delta}_1$-definable.  
\end{lemma}

\begin{proof}
 Work in $V[G]$ and fix  a $\mathbf{\Delta}_1$-subset $\calA$ of $\POT{\kappa}$. We can then find $C\subseteq\kappa$ and $\Sigma_0$-formulas $\varphi_0(v_0,v_1,v_2)$ and $\varphi_1(v_0,v_1,v_2)$ such that $$\calA ~ = ~ \Set{A\subseteq\kappa}{\exists B\subseteq\kappa ~ \varphi_0(A,B,C)}$$ and $$\POT{\kappa}\setminus\calA ~ = ~ \Set{A\subseteq\kappa}{\exists B\subseteq\kappa ~ \varphi_1(A,B,C)}.$$ 
 Assume, towards a contradiction, that $$\mathbbm{1}_{\Add{\kappa}{1}}\not\Vdash\anf{\forall A\subseteq\check{\kappa} ~ \exists B\subseteq\check{\kappa} ~ [\varphi_0(A,B,\check{C})\vee \varphi_1(A,B,\check{C})]}.$$ 
 Since $\Add{\kappa}{1}$ satisfies the $\kappa^+$-chain condition, we can now find a condition $s$ in $\Add{\kappa}{1}$ and an $\Add{\kappa}{1}$-name $\dot{A}$ for a subset of $\kappa$ in $\HH{\kappa^+}$ with 
  \begin{equation}\label{equation:NegDelta}
   s\Vdash_{\Add{\kappa}{1}}\anf{\forall B\subseteq\check{\kappa} ~ [\neg\varphi_0(\dot{A},B,\check{C})\wedge\neg\varphi_1(\dot{A},B,\check{C})]}.
  \end{equation}
Then, there exists an inner model $M$ of $V[G]$ such that $\dot{A},C\in M$ and $V[G]$ is an $\Add{\kappa}{\kappa^+}$-generic extension of $M$. 
  In this situation, there is  $H\in V[G]$ such that $H$ is $\Add{\kappa}{1}$-generic over $M$, $s\in H$ and $V[G]$ is an $\Add{\kappa}{\kappa^+}$-generic extension of $M[H]$. 
  We can now find  $i<2$ and $\bar{B}\in\POT{\kappa}^{V[G]}$ such that $\varphi_i(\dot{A}^H,\bar{B},C)$ holds in $V[G]$. 
  By Lemma \ref{lemma:Silver}, we can find $B\in\POT{\kappa}^{M[H]}$ such that $\varphi(\dot{A}^H,B,C)$ holds in $M[H]$. 
  These computations  yield an extension $t$ of $s$ in $\Add{\kappa}{1}$ such that  
  \begin{equation}\label{equation:ForceSigma}
    t\Vdash_{\Add{\kappa}{1}}\anf{\exists B\subseteq\check{\kappa} ~ \varphi_i(\dot{A},B,\check{C})}.
   \end{equation}
   holds in $M$. Since \eqref{equation:ForceSigma} can be expressed by a $\Sigma_1$-formula, it follows that this statement also holds in $V[G]$, contradicting the fact that \eqref{equation:NegDelta} holds in $V[G]$.  
\end{proof}

 The above lemma now allows us to use class forcing to construct a model of $\ZFC$ in which  non-stationary ideals at uncountable regular cardinals are not $\mathbf{\Delta}_1$-definable.

 \begin{theorem}
  Assume that the  $\GCH$ holds in the ground model $V$ and $V[G]$ is a class forcing extension of $V$ obtained through the standard $\Ord$-length forcing iteration with Easton support that adds  $\kappa^+$-many Cohen subsets to every infinite regular cardinal $\kappa$. 
  Then $V[G]$ is a cofinality-preserving class forcing extension of $V$ and, in $V[G]$, if $\kappa$ is an infinite regular cardinal, then every  $\mathbf{\Delta}_1$-subset of $\POT{\kappa}$ has the $\kappa$-Baire property.    
 \end{theorem}
 
 \begin{proof}
  Standard factoring arguments (see {\cite[p. 233]{MR1940513}}) show that for every infinite regular cardinal $\kappa$ in $V$, there are inner models $M$ and $N$ of $V[G]$ with the property that $V\subseteq M\subseteq N$, $M$ is a cofinality-preserving extension of $V$, $N$ is an $\Add{\kappa}{\kappa^+}$-generic extension of $M$ and $\HH{\kappa^+}^{V[G]}\subseteq N$. 
  The conclusion of the theorem then follows from applications of Lemma \ref{lemma:AbsoBaireProperty} and Lemma \ref{lemma:AddDeltaAbsoDelta}.  
 \end{proof}

 The above theorem shows that, if $\ZFC$ is consistent, then this theory does not prove that there is an uncountable regular cardinal $\kappa$ with the property that there exists a $\mathbf{\Delta}_1$-subset of $\POT{\kappa}$ that separates  $\Club{\kappa}$ from $\NS{\kappa}$. Moreover, since the class forcing used in the above proof is known to preserve various large cardinals, the above argument shows that the existence of an uncountable regular cardinal $\kappa$ with the property that  there exists a $\mathbf{\Delta}_1$-subset of $\POT{\kappa}$ that separates  $\Club{\kappa}$ from $\NS{\kappa}$  is also not provable from extensions of $\ZFC$ by large cardinal axioms.


\subsection{Large cardinals}

In this section, we will present an argument due to S. Friedman and Wu in \cite{FriedmanWu} that shows that sufficiently strong large cardinal properties of a cardinal $\kappa$ ensure that $\Club{\kappa}$ is not $\mathbf{\Delta}_1$-definable. 
 We will present a proof of this result that makes use of classical ideas from  descriptive set theory.

\begin{definition}
  Given a collection $\Gamma$ of subsets of a topological space $\calX$, we say that an element $\calA$ of $\Gamma$ is \emph{complete for $\Gamma$} if for every element $\calB$ of $\Gamma$, there is a continuous function $\map{f}{\calX}{\calX}$ with $\calB=f^{{-}1}[\calA]$. 
\end{definition}

\begin{proposition}\label{proposition:DST}
 If  $\Gamma$ is  a collection of subsets of a topological space $\calX$ that is closed under preimages of continuous functions from $\calX$ to $\calX$  and $\calA$ is complete for $\Gamma$ with the property that $\calX\setminus\calA\in\Gamma$, then $\Gamma$ is closed under complements in $\calX$. 
\end{proposition}

\begin{proof}
 Pick $\calB\in\Gamma$. Then there is a continuous function $\map{f}{\calX}{\calX}$ with $\calB=f^{{-}1}[\calA]$. In this situation, we know that $\calX\setminus\calB=f^{{-}1}[\calX\setminus\calA]$ and $\calX\setminus\calA\in\Gamma$. Hence, the closure of $\Gamma$ under continuous preimages ensures that $\calX\setminus\calB\in\Gamma$.  
\end{proof}

We now prove two lemmata that allow us to apply the above proposition to the collection of all subsets of the power set of an inaccessible cardinal that are definable by $\Sigma_1$-formulas with parameters in $\HH{\kappa^+}$.

\begin{lemma}\label{lemma:Preimage}
 Given an infinite cardinal $\kappa$ satisfying $\kappa=\kappa^{{<}\kappa}$, the collection of subsets of $\POT{\kappa}$ that are  definable by  $\Sigma_1$-formulas with parameters in $\HH{\kappa^+}$  is closed under  continuous preimages. 
\end{lemma}

\begin{proof}
 Fix a subset $\calA$ of $\POT{\kappa}$ that is definable by a $\Sigma_1$-formula with parameters in $\HH{\kappa^+}$ and a continuous function $\map{f}{\POT{\kappa}}{\POT{\kappa}}$. Let $B$ denote the set of all quadruples $\langle\beta,b,\gamma,c\rangle$ satisfying $\beta,\gamma<\kappa$, $b\subseteq\beta$, $c\subseteq\gamma$ and $\calN_{\beta,b}\subseteq f^{{-}1}[\calN_{\gamma,c}]$. 
 Given $X\subseteq\kappa$, we then have $X\in f^{{-}1}[\calA]$ if and only if there is $Y\in\calA$ with the property that for all $\gamma<\kappa$, there is $\beta<\kappa$ with $\langle\beta,X\cap\beta,\gamma,Y\cap\gamma\rangle\in B$. Since the assumption  $\kappa=\kappa^{{<}\kappa}$ implies that $B$ is an element of $\HH{\kappa^+}$, we can conclude that the set $f^{{-}1}[\calA]$ is also definable by a $\Sigma_1$-formula with parameters in $\HH{\kappa^+}$. 
\end{proof}

\begin{lemma}\label{lemma:Complements}
  Given an infinite cardinal $\kappa$, the collection of subsets of $\POT{\kappa}$ that are definable by $\Sigma_1$-formulas with parameters in $\HH{\kappa^+}$ is not closed under complements in $\POT{\kappa}$. 
\end{lemma} 
 
 \begin{proof}
 Pick  a \emph{universal $\Sigma_1$-formula $\varphi(v_0,\ldots,v_3)$}, {i.e.,} a $\Sigma_1$-formula $\varphi(v_0,\ldots,v_3)$ with the property that for every $\Sigma_1$-formula $\psi(v_0,v_1,v_2)$, there is a natural number $m$ such that the sentence 
 \begin{equation}\label{equation:UniversalFormula}
  \forall x,y,z ~ [\varphi(x,y,z,m)\longleftrightarrow\psi(x,y,z)]
 \end{equation}
 is provable in $\ZFC^-$ (see {\cite[Section 1]{JensenFine}}). Then the set $$\calA ~ = ~ \Set{n\cup\Set{\omega+\alpha}{\alpha\in X}}{\textit{$X\subseteq\kappa$ and $n<\omega$ with $\varphi(X,n,X,n)$}} ~ \subseteq ~ \POT{\kappa}$$ is definable by a $\Sigma_1$-formula with parameters in $\HH{\kappa^+}$. 
  Assume, towards a contradiction, that $\POT{\kappa}\setminus\calA$ is also definable by a $\Sigma_1$-formula with parameters in $\HH{\kappa^+}$. 
  Then we can find a $\Sigma_1$-formula $\psi(v_0,v_1,v_2)$ and  a subset $Y$ of $\kappa$ such that $$\psi(X,n,Y) ~ \longleftrightarrow ~ \neg\varphi(X,n,X,n)$$ holds for all $X\subseteq\kappa$ and $n<\omega$. 
 Now, pick a natural number $m$ such that \eqref{equation:UniversalFormula} holds with respect to this $\Sigma_1$-formula $\psi$. 
 We can then conclude that $$\varphi(Y,m,Y,m) ~ \longleftrightarrow ~ \neg\varphi(Y,m,Y,m),$$ a contradiction. 
\end{proof}

 Remember that a cardinal $\kappa$ is \emph{weakly compact} if and only if it is $\Pi^1_1$-indescribable, {i.e.,} if it has the property that for every $R\subseteq V_\kappa$ and every $\Pi^1_1$-sentence $\Phi$ with $\langle V_\kappa,\in,R\rangle\models\Phi$, there is an $\alpha<\kappa$ with the property that $\langle V_\alpha,\in,R\cap V_\alpha\rangle\models\Phi$. 
Following \cite{MR0281606}, we associate a filter to this reflection property:

\begin{definition}
 Given a weakly compact $\kappa$, the \emph{weakly compact filter} on $\kappa$ consists of all subsets $X$ of $\kappa$ with the property that there is a $\Pi^1_1$-sentence $\Psi$ and $R\subseteq V_\kappa$ with $\langle V_\kappa,\in,R\rangle\models\Phi$ and $\Set{\alpha<\kappa}{\langle V_\alpha,\in,R\cap V_\alpha\rangle\models\Phi} ~ \subseteq ~ X$. 
\end{definition}

A short argument then shows that the weakly compact filter is a normal filter on the given weakly compact cardinal (see \cite{MR0281606}). In particular, all elements of this collection are stationary. Moreover, it is easy to see that this filter contains all sets of the form $\Refl{S}$ for stationary subsets $S$ of the given weakly compact cardinal.

\begin{theorem}\label{theorem:WC}
 If $\kappa$ is a weakly compact cardinal and $S$ is an element of the weakly compact filter on $\kappa$, then $\Club{S}$ is  complete for the collection of subsets of $\POT{\kappa}$ that are definable by $\Sigma_1$-formulas with parameters in $\HH{\kappa^+}$. 
\end{theorem}

\begin{proof}
 Let $\calA$ be a subset of $\POT{\kappa}$ that is definable by a $\Sigma_1$-formula with parameters in $\HH{\kappa^+}$. Then we can  find a $\Sigma_1$-formula $\varphi(v_0,v_1)$ and $Y\subseteq\kappa$ with the property that $\calA=\Set{X\subseteq\kappa}{\varphi(X,Y)}$. Define $$\Map{f}{\POT{\kappa}}{\POT{\kappa}}{X}{\Set{\alpha\in S}{\varphi(X\cap\alpha,Y\cap\alpha)}}.$$ 
 It is then easy to see that the function $f$ is \emph{Lipschitz}, {i.e.,} $f(X_0)\cap\alpha=f(X_1)\cap\alpha$ holds for all $\alpha<\kappa$ and all $X_0,X_1\subseteq\kappa$ with $X_0\cap\alpha=X_1\cap\alpha$. In particular, it follows that $f$ is a continuous function.

 \begin{claim*}
  If $X\in\calA$, then $f(X)\in\Club{S}$. 
 \end{claim*}
 
 \begin{proof}[Proof of the Claim]
  Pick an increasing and continuous chain\footnote{In the following, we say that a sequence $\seq{B_\alpha}{\alpha<\lambda}$ of sets is an increasing and continuous chain if $B_\alpha\subseteq B_\beta$ holds for all $\alpha<\beta<\lambda$ and $B_\beta=\bigcup\Set{B_\alpha}{\alpha<\beta}$ holds for all limit ordinals $\beta<\lambda$.} $\seq{N_\alpha}{\alpha<\kappa}$ of elementary submodels of $\HH{\kappa^+}$ with $X,Y\in N_0$ and $\alpha\subseteq N_\alpha\cap\kappa\in \kappa$ for all $\alpha<\kappa$. Then $C=\Set{N_\alpha\cap\kappa}{\alpha<\kappa}$ is a closed unbounded subset of $\kappa$. Fix $\alpha\in C\cap S$, pick $\beta<\kappa$ with $N_\beta\cap\kappa=\alpha$ and let $\map{\pi}{N_\beta}{M}$ denote the corresponding transitive collapse. Then $\pi(\kappa)=\alpha$, $\pi(X)=X\cap\alpha$ and $\pi(Y)=Y\cap\alpha$. Moreover, since our setup ensures that $\varphi(X,Y)$ holds in $N_\beta$, we know that $\varphi(X\cap\alpha,Y\cap\alpha)$ holds in $M$. Using $\Sigma_1$-upwards absoluteness, we can now conclude that $\alpha\in f(X)$. In particular, it follows that $C$ witnesses that $f(X)$ is an element of $\Club{S}$.  
 \end{proof}

 \begin{claim*}
  If $X\subseteq\kappa$ with $f(X)\in\Club{S}$, then $X\in\calA$. 
 \end{claim*}
 
 \begin{proof}[Proof of the Claim]
  Assume, towards a contradiction, that $\varphi(X,Y)$ does not hold. 
  Since $S$ is contained in the weakly compact filter on $\kappa$, there exists a $\Pi^1_1$-sentence $\Psi$ and a subset $R$ of $V_\kappa$ with the property that $\langle V_\kappa,\in,R\rangle\models\Psi$ and $$\Set{\alpha<\kappa}{\langle V_\alpha,\in,R\cap V_\alpha\rangle\models\Psi} ~ \subseteq ~ S.$$ In addition, we can find a closed unbounded subset $C$ of $\kappa$ with the property that $C\cap S\subseteq f(X)$. 
  Using {\cite[Lemma 4.2 \& Corollary 4.3]{MR3913154}}, we can now find a limit cardinal $\theta>\kappa$, a transitive set $M$, an inaccessible cardinal $\bar{\kappa}\in M\cap\kappa$ with $\HH{\bar{\kappa}^+}^M\prec_{\Sigma_1}\HH{\bar{\kappa}^+}$  and a non-trivial elementary embedding $\map{j}{M}{\HH{\theta}}$ with critical point $\bar{\kappa}$, $j(\bar{\kappa})=\kappa$ and $C,R,X,Y\in\ran{j}$. 
  We then know that $\bar{\kappa}\in C$ and $V_{\bar{\kappa}},R\cap V_{\bar{\kappa}},X\cap\bar{\kappa},Y\cap\bar{\kappa}\in M$ with $j(V_{\bar{\kappa}})=V_\kappa$, $j(R\cap V_{\bar{\kappa}})=R$, $j(X\cap\bar{\kappa})=X$ and $j(Y\cap\bar{\kappa})=Y$. 
 Our setup then allows us to use $\Sigma_1$-upwards absoluteness to show that both $\neg\varphi(X\cap\bar{\kappa},Y\cap\bar{\kappa})$ and $\langle V_{\bar{\kappa}},\in,R\cap V_{\bar{\kappa}}\rangle\models\Psi$ hold. But, this allows us to conclude that $\bar{\kappa}\in C\cap S\subseteq f(X)$, contradicting the definition of $f(C)$.  
 \end{proof}

 Since the above claims show that for every subset $\calA$ of $\POT{\kappa}$ that is definable by a $\Sigma_1$-formula with parameters in $\HH{\kappa^+}$, there exists a continuous function $\map{f}{\POT{\kappa}}{\POT{\kappa}}$ with $\calA=f^{{-}1}[\Club{S}]$, we know that $\Club{S}$  is  complete for the collection of subsets of $\POT{\kappa}$ that are definable in this way. 
\end{proof}

We are now ready to show how large cardinal properties of $\kappa$ answer Question \ref{question:Q1} at $\kappa$. The following result slightly strengthens {\cite[Proposition 2.1]{FriedmanWu}}:

\begin{corollary}
 If $\kappa$ is a weakly compact cardinal and $S$ is an element of the weakly compact filter on $\kappa$, then $\Club{S}$ is not a $\mathbf{\Delta}_1$-subset of $\POT{\kappa}$. 
\end{corollary}

\begin{proof}
  Assume, towards a contradiction, that $\POT{\kappa}\setminus\Club{S}$ is definable by a $\Sigma_1$-formula with parameters in $\HH{\kappa^+}$. 
  Since $\kappa$ is inaccessible, we can then apply Lemma \ref{lemma:Preimage} to show  that the collection of subsets of $\POT{\kappa}$ that are definable by $\Sigma_1$-formulas with parameters in $\HH{\kappa^+}$ is closed under continuous preimages. 
  Moreover, since Theorem \ref{theorem:WC} ensures that $\Club{S}$ is complete for the collection of all subsets of $\POT{\kappa}$ that are definable by $\Sigma_1$-formulas with parameters in $\HH{\kappa^+}$, 
 Proposition \ref{proposition:DST}  implies that this collection is closed under complements, directly contradicting Lemma \ref{lemma:Complements}.  
\end{proof}

\begin{corollary}\label{corollary:NonDeltaWC}
  If $\kappa$ is a weakly compact cardinal, then $\Club{\kappa}$  is not a $\mathbf{\Delta}_1$-subset of $\POT{\kappa}$. \qed  
\end{corollary}


\subsection{Condensation}

The aim of this section is to show that, in canonical inner models, the non-stationary ideal of an uncountable regular cardinal is not $\mathbf{\Delta}_1$-definable. 
 In the case of the constructible universe $L$, this conclusion was proven by S. Friedman, Hyttinen and Kulikov (see {\cite[Theorem 49.(3)]{MR3235820}}). In the following, we will show that the key idea from their proof can also be used to prove analogous results for various other canonical inner models. 
 These arguments will be based on an abstract condensation principle introduced in the next definition:

\begin{definition}
 Given an uncountable regular cardinal $\kappa$ and a subset $S$ of $\kappa$, we let $\Cond{S}$ denote the statement that there exists an  increasing and continuous  sequence $\seq{M_\beta}{\beta<\kappa}$ of transitive sets 
 with the property that for every $z\in\HH{\kappa^+}$, there exists an increasing and continuous sequence $\seq{N_\varepsilon}{\varepsilon<\kappa}$ of elementary submodels of $\HH{\kappa^+}$ of cardinality less than $\kappa$ such that the following statements hold:  
 \begin{enumerate}
   \item $z\in N_0$ and  $\varepsilon\subseteq N_\varepsilon\cap\kappa\in\kappa$  for all $\varepsilon<\kappa$. 
   
   \item  If $\varepsilon<\kappa$ satisfies   $N_\varepsilon\cap\kappa\in S$, then there is $\beta<\kappa$ such that the following statements hold: 
     \begin{enumerate}
       \item  The transitive collapse of $N_\varepsilon$ is equal to $M_\beta$. 
       
       \item  If $\beta<\gamma<\kappa$ has the property that $M_\gamma$ is a model of $\ZFC^-$, then the set $\Set{N_\delta\cap\kappa}{\delta<\varepsilon}$ is an element of $M_\gamma$. 
     \end{enumerate}
 \end{enumerate}
\end{definition}

In order to  motivate the formulation of the principle $\Cond{\kappa}$, we briefly show that this principle holds at all uncountable regular cardinals in G\"odel's constructible universe. The argument proving  this observation already contains the key idea from the proof of {\cite[Theorem 49.(3)]{MR3235820}}.

\begin{proposition}
 If $V=L$ and $\kappa$ is an uncountable regular cardinal, then $\seq{L_\varepsilon}{\varepsilon<\kappa}$ witnesses that $\Cond{\kappa}$ holds. 
\end{proposition}

\begin{proof}
 Fix $z\in L_{\kappa^+}$ and let $C$ denote the set of all $\alpha<\kappa$ with the property that the Skolem hull $\mathcal{Hull}_{L_{\kappa^+}}(\alpha\cup\{z\})$ (using the canonical Skolem functions in $L$) of $\alpha\cup\{z\}$ in $L_{\kappa^+}$ satisfies $\mathcal{Hull}_{L_{\kappa^+}}(\alpha\cup\{z\})\cap\kappa=\alpha$. Then $C$ is a closed unbounded subset of $\kappa$. Let $\seq{\alpha_\varepsilon}{\varepsilon<\kappa}$ denote the monotone enumeration of $C$ and, for each $\varepsilon<\kappa$, we set $N_\varepsilon=\mathcal{Hull}_{L_{\kappa^+}}(\alpha_\varepsilon\cup\{z\})$.  
  Then $z\in N_0$ and for all $\varepsilon<\kappa$, we know that $N_\varepsilon$ is an elementary submodel of $\HH{\kappa^+}$ of cardinality less than $\kappa$ with $\varepsilon\subseteq\alpha_\varepsilon=N_\varepsilon\cap\kappa\in\kappa$.   Moreover, for each $\varepsilon<\kappa$, there is $\beta_\varepsilon<\kappa$ such that $N_\varepsilon$ collapses to $L_{\beta_\varepsilon}$. 
   Finally, if $\varepsilon<\kappa$, $\beta_\varepsilon<\gamma<\kappa$ with the property that $L_\gamma$ is a model of $\ZFC^-$ and $\map{\pi}{N_\varepsilon}{L_{\beta_\varepsilon}}$ denotes the corresponding collapsing map, then we have $$C\cap\alpha_\varepsilon ~ = ~ \Set{\alpha<\alpha_\varepsilon}{\mathcal{Hull}_{L_{\beta_\varepsilon}}(\alpha\cup\{\pi(z)\})\cap\alpha_\varepsilon ~ = ~ \alpha} ~ \in ~ L_{\gamma}$$ and hence we know that $\Set{N_\delta\cap\kappa}{\delta<\varepsilon}\in L_\gamma$. 
\end{proof}

Next, we observe that the above principle implies fragments of the $\GCH$:

\begin{lemma}\label{lemma:CondGCH}
 If $\kappa$ is an uncountable regular cardinal such that $\Cond{S}$ holds for some stationary subset $S$ of $\kappa$, then $\kappa=\kappa^{{<}\kappa}$ holds. 
\end{lemma}

\begin{proof}
 Let $\seq{M_\beta}{\beta<\kappa}$ witness that $\Cond{S}$ holds and pick $z\in\HH{\kappa}$. 
  Then there is  an increasing and continuous sequence $\seq{N_\varepsilon}{\varepsilon<\kappa}$ of elementary submodels of $\HH{\kappa^+}$ of cardinality less than $\kappa$ such that $z\in N_0$ and 
  for all $\varepsilon<\kappa$ with $N_\varepsilon\cap\kappa\in S$,   we have $\varepsilon\subseteq N_\varepsilon\cap\kappa\in\kappa$ and there is $\beta_\varepsilon<\kappa$ such that $N_\varepsilon$ collapses to $M_{\beta_\varepsilon}$. 
  We now know that  $\Set{\beta_\varepsilon}{\textit{$\varepsilon<\kappa$ with $N_\varepsilon\cap\kappa\in S$}}$ is a cofinal subset of $\kappa$ and this implies that $M_\beta$ has  cardinality less than $\kappa$ for all $\beta<\kappa$. 
  Moreover, there is $\varepsilon<\kappa$ such that $N_\varepsilon\cap\kappa\in S$ and there is a surjection of $\varepsilon$ onto the transitive closure of $z$. It then follows that $z\in M_{\beta_\varepsilon}$. 
 
 The above computations show that $\HH{\kappa}$ is equal to the union of $\kappa$-many sets of cardinality less than $\kappa$ and therefore we know that this set has cardinality $\kappa$. This proves that $\kappa=\kappa^{{<}\kappa}$ holds. 
\end{proof}

We now show how the above condensation principle is connected to the complexity of non-stationary ideals.

\begin{theorem}\label{theorem:Condense}
  If $\kappa$ is an uncountable regular cardinal  and $S$ is a stationary subset of $\kappa$  with the property that $\Cond{S}$ holds, then $\Club{S}$ is complete for the collection of subsets of $\POT{\kappa}$ that are definable by $\Sigma_1$-formulas with parameters in $\HH{\kappa^+}$. 
\end{theorem}

\begin{proof}
  Fix a  subset $\calA$ of $\POT{\kappa}$ with the property that $\calA=\Set{X\subseteq\kappa}{\varphi(X,Y)}$ holds  for some $\Sigma_1$-formula $\varphi(v_0,v_1)$ and some $Y\subseteq\kappa$. 
  Let  $\seq{M_\varepsilon}{\varepsilon<\kappa}$ be a sequence of transitive sets witnessing that $\Cond{S}$ holds. 
  Given $X\subseteq\kappa$, we define $f(X)$ to be the set of all limit ordinals $\alpha\in S$ with  the property that there exists $\beta<\kappa$ such that the following statements hold: 
  \begin{itemize}
   \item $M_\beta$ is a model of $\ZFC^-$. 
  
   \item $\alpha,S\cap\alpha,X\cap\alpha,Y\cap\alpha\in M_\beta$. 
   
   \item $\varphi(X\cap\alpha,Y\cap\alpha)$ holds in $M_\beta$. 
   
   \item $\alpha$ is an uncountable  regular cardinal in $M_\beta$. 
  
   \item $S\cap\alpha$ is a stationary subset of $\alpha$ in $M_\beta$. 
  \end{itemize}
 It is then easy to see that the resulting function $\map{f}{\POT{\kappa}}{\POT{\kappa}}$ is Lipschitz and therefore continuous.

 \begin{claim*}
  If $X\in\calA$, then $f(X)\in\Club{S}$. 
 \end{claim*}
 
 \begin{proof}[Proof of the Claim]
  By our assumptions, there exists an increasing and continuous sequence $\seq{N_\varepsilon}{\varepsilon<\kappa}$ of elementary submodels of $\HH{\kappa^+}$ of cardinality less than $\kappa$ such that $\Lim\cap S,X,Y\in N_0$ and  for every $\varepsilon<\kappa$, we have $\varepsilon\subseteq N_\varepsilon\cap\kappa\in\kappa$ and, if  $N_\varepsilon\cap\kappa\in S$, then there is  $\beta_\varepsilon<\kappa$ with the property that the transitive collapse of $N_\varepsilon$ is equal to $M_{\beta_\varepsilon}$.  
  Set $C=\Set{N_\varepsilon\cap\kappa}{\varepsilon<\kappa}$.  Then $C$ is a closed unbounded subset of $\kappa$ that consists of limit ordinals.  
  Fix  $\alpha\in C\cap S$ and pick $\varepsilon<\kappa$ with $N_\varepsilon\cap\kappa=\alpha$. Then the fact that $M_{\beta_\varepsilon}$ is  the transitive collapse of $N_\varepsilon$ implies that $M_{\beta_\varepsilon}$ is a model $\ZFC^-$, $\alpha,S\cap\alpha,X\cap\alpha,Y\cap\alpha\in M_{\beta_\varepsilon}$, $\varphi(X\cap\alpha,Y\cap\alpha)$ holds in $M_{\beta_\varepsilon}$, $\alpha$ is a regular uncountable cardinal in $M_{\beta_\varepsilon}$ and $\Lim\cap S\cap\alpha$ is a stationary subset of $\alpha$ in $M_{\beta_\varepsilon}$. This shows that $\beta_\varepsilon$ witnesses that $\alpha$ is an element of $f(X)$.  
 \end{proof}

  \begin{claim*}
  If $X\subseteq\kappa$ with $f(X)\in\Club{S}$, then $X\in\calA$.  
 \end{claim*}
 
 \begin{proof}[Proof of the Claim]
  Assume, towards a contradiction, that $\varphi(X,Y)$ does not hold. 
  Pick a closed unbounded subset $C_0$ of $\kappa$ with $C_0\cap S\subseteq f(X)$. 
   Our assumptions now allow us to find an increasing and continuous sequence $\seq{N_\varepsilon}{\varepsilon<\kappa}$ of elementary submodels of $\HH{\kappa^+}$ of cardinality less than $\kappa$ such that $C_0,S,X,Y\in N_0$ and  for every $\varepsilon<\kappa$, we have $\varepsilon\subseteq N_\varepsilon\cap\kappa\in\kappa$ and, if $N_\varepsilon\cap\kappa\in S$, then there is $\beta_\varepsilon<\kappa$ such that the transitive collapse of $N_\varepsilon$ is equal to $M_{\beta_\varepsilon}$ and, if $\beta_\varepsilon<\gamma<\kappa$ has the property that $M_\gamma$ is a model of $\ZFC^-$, then $\Set{N_\delta\cap\kappa}{\delta<\varepsilon}\in M_\gamma$. 
   It then follows that   $C_1=\Set{N_\varepsilon\cap\kappa}{\varepsilon<\kappa}$ is a closed unbounded subset of $\kappa$ and  $S$ contains a limit point of $C_1$. 
   Set $\alpha=\min(\Lim(C_1)\cap S)$ and  
   pick $\varepsilon<\kappa$ with $N_\varepsilon\cap\kappa=\alpha$.  
   We then know that the transitive collapse of $N_\varepsilon$ is equal to $M_{\beta\varepsilon}$ and this implies that $X\cap\alpha,Y\cap\alpha\in M_{\beta_\varepsilon}$ and $\varphi(X\cap\alpha,Y\cap\alpha)$ fails in $M_{\beta_\varepsilon}$. 
   Moreover, this setup ensures that $\alpha$ is a limit point of $C_0$ and we can conclude that $\alpha\in C_0\cap S\subseteq f(X)$ holds. 
   Therefore, we can find $\gamma<\kappa$ such that $M_\gamma$ is a model of $\ZFC^-$, $\alpha,S\cap\alpha,X\cap\alpha,Y\cap\alpha\in M_\gamma$, $\varphi(X\cap\alpha,Y\cap\alpha)$ holds in $M_\gamma$, $\alpha$ is an uncountable  regular cardinal in $M_\gamma$ and $S\cap\alpha$ is a stationary subset of $\alpha$ in $M_\gamma$. 
   Since $\varphi(X\cap\alpha,Y\cap\alpha)$ holds in $M_\gamma$ and fails in $M_{\beta_\varepsilon}$, $\Sigma_1$-upwards absoluteness implies that $M_\gamma\nsubseteq M_{\beta_\varepsilon}$ and hence we know that $\beta_\varepsilon<\gamma<\kappa$. 
 But, this implies that $C_1\cap\alpha=\Set{N_\delta\cap\kappa}{\delta<\varepsilon}\in M_\gamma$ and this set witnesses that either $\alpha$ is not an uncountable regular cardinal in $M_\gamma$ or $S\cap\alpha$ is not a stationary subset of $\alpha$ in $M_\gamma$, a contradiction.  
\end{proof}

 The above claims complete the proof of the theorem. 
\end{proof}

We are now ready to prove the following slight strengthening of {\cite[Theorem 49.(3)]{MR3235820}}:

\begin{corollary}
  If $\kappa$ is an uncountable regular cardinal  and $S$ is a stationary subset of $\kappa$ with the property that $\Cond{S}$ holds,  then $\Club{S}$ is not a $\mathbf{\Delta}_1$-subset of $\POT{\kappa}$.  
\end{corollary}

\begin{proof}
  A combination of Lemma \ref{lemma:Preimage}, Lemma \ref{lemma:Complements} and Lemma \ref{lemma:CondGCH} shows that the  collection of subsets of $\POT{\kappa}$ that are definable by $\Sigma_1$-formulas with parameters in $\HH{\kappa^+}$  is closed under continuous preimages and not closed under complements. 
 Since Theorem \ref{theorem:Condense} ensures that $\Club{S}$ is complete for this collection, we can now apply  Proposition \ref{proposition:DST} to conclude that $\POT{\kappa}\setminus\Club{S}$ is not a $\mathbf{\Delta}_1$-subset of $\POT{\kappa}$.  
\end{proof}

In the remainder of this section, we will use results on abstract condensation principles contained in  \cite{fernandes2021local}, \cite{MR2860182},  \cite{MR3206451} and \cite{MR3436373} to show that, in various canonical inner models, the principle $\Cond{\kappa}$ holds for every uncountable regular cardinal $\kappa$.

\begin{definition}[\cite{MR2860182,MR3436373}]
 Given an uncountable cardinal $\kappa$, we say that \emph{Local Club Condensation holds at $\kappa^+$} if there exists a sequence $\seq{M_\gamma}{\gamma\leq\kappa^+}$ of transitive sets such that the following statements hold: 
  \begin{enumerate}
    \item If $\beta<\gamma\leq\kappa^+$, then $M_\beta\in M_\gamma$ and $M_\beta\cap\Ord=\beta$. 
    
    \item If $\gamma\leq\kappa^+$ is a limit ordinal, then $M_\gamma=\bigcup\Set{M_\beta}{\beta<\gamma}$. 
    
    \item $M_{\kappa^+}=\HH{\kappa^+}$. 
    
    \item If $\kappa\leq\gamma<\kappa^+$ and $\mathbb{M}$ is a structure in a countable language that expands the structure $\langle M_\gamma,\in,\seq{M_\beta}{\beta<\gamma}\rangle$, then there is an increasing and continuous sequence $\seq{\mathbb{N}_\alpha}{\alpha<\kappa}$ of elementary submodels of $\mathbb{M}$ of cardinality less than $\kappa$     satisfying the following statements, where $N_\alpha$ denotes the underlying set of $\mathbb{N}_\alpha$ for each $\alpha<\kappa$: 
      \begin{enumerate}
       \item $\alpha\subseteq N_\alpha$ for all $\alpha<\kappa$. 
       
       \item $M_\gamma=\bigcup\Set{N_\alpha}{\alpha<\kappa}$. 
       
       \item If $\alpha<\kappa$, then there is $\beta<\kappa$ with the property that the transitive collapse of $N_\alpha$ is equal to $M_\beta$. 
      \end{enumerate}
  \end{enumerate}
\end{definition}

It is known that, in various canonical inner models of the form $L[E]$, the sequences $\seq{L_\gamma[E]}{\gamma\leq\kappa^+}$ witness that Local Club Condensation holds at uncountable cardinals $\kappa^+$  (see \cite{fernandes2021local} and {\cite[Theorem 8]{MR2860182}}). Therefore, the following result shows that the non-stationary ideal of an uncountable regular cardinal is not $\mathbf{\Delta}_1$-definable in these models.

\begin{lemma}
 If $\kappa$ is an uncountable regular cardinal and $A$ is a subset of $\kappa^+$ with the property that the sequence $\seq{L_\gamma[A]}{\gamma\leq\kappa^+}$ witnesses that Local Club Condensation holds at $\kappa^+$, then $\seq{L_\beta[A]}{\beta<\kappa}$ witnesses that $\Cond{\kappa}$ holds. 
\end{lemma}

\begin{proof}
 Fix $z\in\HH{\kappa^+}$ and consider the structure $$\mathbb{M} ~ = ~ \langle\HH{\kappa^+},\in,\seq{L_\gamma[A]}{\gamma<\kappa^+},F,S\rangle,$$ where $S$ is a set of Skolem functions for the given structure that are defined using the canonical well-ordering $<_{L[A]}$ of $L[A]$ and $F=\seq{f_\gamma}{\kappa\leq\gamma<\kappa^+}$ is the unique sequence with the property that  $f_\gamma$ is the $<_{L[A]}$-least bijection between $\kappa$ and $\gamma$ for all $\kappa\leq\gamma<\kappa^+$.  
  Let $C$ denote the set of all $\alpha<\kappa$ with the property that $\alpha$ is equal to the intersection of $\kappa$ with $\mathcal{Hull}_{\mathbb{M}}(\alpha\cup\{z\})$. Then $C$ is a closed unbounded subset of $\kappa$. Let $\seq{\alpha_\varepsilon}{\varepsilon<\kappa}$ denote its monotone enumeration. 
  Given $\varepsilon<\kappa$, set $N_\varepsilon=\mathcal{Hull}_{\mathbb{M}}(\alpha_\varepsilon\cup\{z\})$ and let $\mathbb{N}_\varepsilon$ denote the substructure of $\mathbb{M}$ with underlying set $N_\varepsilon$. 
  In addition, for each $\varepsilon<\kappa$, we let $\mathbb{M}_\varepsilon$ denote the unique structure in the same language as $\mathbb{M}$ that has the transitive collapse of $N_\varepsilon$ as its  underlying set and is isomorphic to $\mathbb{N}_\varepsilon$ via the corresponding uncollapsing map.

  Then $\seq{N_\varepsilon}{\varepsilon<\kappa}$ is an increasing and continuous sequence of elementary substructures of $\HH{\kappa^+}$ with $z\in N_0$ and $\varepsilon\subseteq N_\varepsilon\cap\kappa\in\kappa$ for all $\varepsilon<\kappa$. 
  Moreover, if $\varepsilon<\kappa$, then {\cite[Theorem 1]{MR3436373}} ensures that the transitive collapse of $N_\varepsilon$ is equal to $L_\beta[A]$ for some $\beta<\kappa$.  
  Finally, pick $\varepsilon<\beta<\gamma<\kappa$ with the property that the transitive collapse of $N_\varepsilon$ is equal to $L_\beta[A]$ and  $L_\gamma[A]$ is a model of $\ZFC^-$. 
  We then know that the structure $\mathbb{M}_\varepsilon$ is an element of $L_\gamma[A]$, because $L_\gamma[A]$ contains both $L_\beta[A]$ and $A\cap\beta$ as elements and elementarity ensures that the functions and relations of $\mathbb{M}_\varepsilon$ are all definable over the structure $\langle L_\beta[A],\in,A\rangle$. 
  Moreover, note that $C\cap\alpha_\varepsilon$ is equal to the set of all $\alpha<\alpha_\varepsilon$ with the property that $\alpha$ is equal to the intersection of $\alpha_\varepsilon$ with $\mathcal{Hull}_{\mathbb{M}_\varepsilon}(\alpha\cup\{\bar{z}\})$, where $\bar{z}$ denotes the image of $z$ under the transitive collapse of $N_\varepsilon$. 
  In combination, this shows that $C\cap\alpha_\varepsilon$ is an element of $L_\gamma[A]$. 
\end{proof}


\subsection{Forcing axioms}
 As our final example of a canonical setting in which the non-stationary ideal is not $\mathbf{\Delta}_1$-definable, we present the statement of a result of Hoffelner, Larson, Schindler and Wu in \cite{HOFFELNER_LARSON_SCHINDLER_WU_2023} that shows that strong forcing axioms imply that $\Club{\omega_1}$ is not a $\mathbf{\Delta}_1$-subset of $\POT{\omega_1}$. 
 Remember that \emph{Woodin's axiom $(*)$} postulates that the Axiom of Determinacy holds in $L(\mathbb{R})$ and $L(\POT{\omega_1})$ is a $\PPP_{max}$-extension of $L(\mathbb{R})$ (see {\cite[Chapter 5]{MR1713438}}). The definition of \emph{Bounded Martin's Maximum $\BMM$} can be found in \cite{MR1324501} and {\cite[Section 10.3]{MR1713438}}.

\begin{theorem}[\cite{HOFFELNER_LARSON_SCHINDLER_WU_2023}]\label{theorem:HoffelnerEtAlDelta}
 Assume that either Woodin's axiom $(*)$ holds or there is a Woodin cardinal and $BMM$ holds. Then $\Club{\omega_1}$ is not a $\mathbf{\Delta}_1$-subset of $\POT{\omega_1}$. 
\end{theorem}

 Since Asper\'{o} and Schindler \cite{MR4250390}  proved that the strengthening $\MM^{++}$ of Martin's Maximum (see \cite{MM} and {\cite[Definition 2.45]{MR1713438}}) implies that Woodin's axiom $(*)$ holds, it follows that this forcing axiom provides a negative answer to Question \ref{question:Q1}. In contrast, Hoffelner, Larson, Schindler and Wu proved in \cite{PFA_Delta_1} that the \emph{Proper Forcing Axiom $\PFA$} does not suffice for the above conclusion.


\section*{Acknowledgements}

The author gratefully acknowledges support by the \emph{Deutsche Forschungsgemeinschaft} (Project number 522490605). 
 In addition, he is thankful to Peter Holy for a very helpful discussion on the results of \cite{MR2860182}. 
 Finally, he would like to thank  the anonymous referee for the careful reading of the manuscript and several helpful comments.


 \bibliographystyle{plain}
\bibliography{references}

\end{document}